\documentclass[11pt, a4paper,leqno]{amsart}
\usepackage{amsmath,amsthm,amscd,amssymb,amsfonts, amsbsy}
\usepackage{latexsym}
\usepackage{txfonts}
\usepackage{exscale}
\usepackage{mathrsfs}

\usepackage{huncial}
\usepackage[T1]{fontenc}

\usepackage[colorlinks,citecolor=red,pagebackref,hypertexnames=false, breaklinks]{hyperref}

\calclayout
\allowdisplaybreaks

\def\Xint#1{\mathchoice
    {\XXint\displaystyle\textstyle{#1}}%
    {\XXint\textstyle\scriptstyle{#1}}%
    {\XXint\scriptstyle\scriptscriptstyle{#1}}%
    {\XXint\scriptscriptstyle\scriptscriptstyle{#1}}%
    \!\int}
    \def\XXint#1#2#3{{\setbox0=\hbox{$#1{#2#3}{\int}$}
    \vcenter{\hbox{$#2#3$}}\kern-.5\wd0}}
    \def\dashint{\Xint-}

\renewcommand{\chi}{{\bf 1}}

\theoremstyle{plain}
\newtheorem{theorem}[equation]{Theorem}
\newtheorem{lemma}[equation]{Lemma}
\newtheorem{corollary}[equation]{Corollary}
\newtheorem{proposition}[equation]{Proposition}

\theoremstyle{definition}

\theoremstyle{remark}

\numberwithin{equation}{section}

\numberwithin{equation}{section}

\def \R{ \mathbb{R} }
\def \N{ \mathbb{N} }
\def \Z{ \mathbb{Z} }
\def \C{ \mathbb{C} }
\def \iint{\int\!\!\!\int}
\def\div{\mathop{\rm div}}
\renewcommand{\Re}{{\rm Re}\,}

\DeclareMathOperator{\supp}{supp}

\begin{document}

\allowdisplaybreaks

\title[Tent spaces and extrapolation]{Tent space boundedness via extrapolation}

\author{Pascal Auscher}
\address{Pascal Auscher, Laboratoire de Math\'{e}matiques d'Orsay, Univ. Paris-Sud, CNRS, Universit\'{e} Paris-Saclay, 91405 Orsay, France}
\email{Pascal.Auscher@math.u-psud.fr}

\author{Cruz Prisuelos-Arribas}

\address{Cruz Prisuelos-Arribas
Instituto de Ciencias Matem\'aticas (CSIC-UAM-UC3M-UCM)
\\
C/ Nicol\'as Cabrera, 13-15
\\
E-28049 Madrid, Spain} \email{cruz.prisuelos@icmat.es}

\thanks{The first author was partially supported by the ANR project ``Harmonic Analysis at its Boundaries'', ANR-12-BS01-0013. The second author has been supported by the European Research
Council under the European Union's Seventh Framework Programme (FP7/2007-2013)/ ERC
agreement no. 615112 HAPDEGMT}

\date{March 3, 2016}


\begin{abstract} We study the action of operators on tent spaces such as maximal operators, Calder\'on-Zygmund operators, Riesz potentials. We also consider singular non-integral operators. We obtain boundedness as an application of extrapolation methods in the Banach range.  In the non Banach range, boundedness results for Calder\'on-Zygmund operators follows by using an appropriate atomic theory. We end with some consequences on amalgalm spaces. \end{abstract}

\maketitle

\tableofcontents

\bigskip

\section{Introduction}

For a measurable function $F:\R^{n+1}_+:=\R^n\times (0,\infty) \to \C$ and $0<r<\infty$, let \begin{equation}
\label{eq:A}
\mathcal{A}_r(F)(x):=\left(\int_0^{\infty}\int_{B(x,t)}|F(y,t)|^r\frac{dy\,dt}{t^{n+1}}\right)^{\frac{1}{r}}, \quad x\in \R^n.
\end{equation}

Consider the tent space $T_r^q$, $0<q,r<\infty$, defined as the space of all measurable functions $F$ such that $\mathcal{A}_r(F)\in L^q(\R^n)$. We also define the weak tent space $wT^q_r$ as the space of all measurable functions $F$ such that $\mathcal{A}_r(F)\in L^{q,\infty}(\R^n)$.

These spaces play an important role in harmonic analysis as evidenced in \cite{CoifmanMeyerStein}, starting from the use of Lusin area functional on harmonic functions. 
They are heavily used in the recent theory of Hardy spaces associated with operators (\cite{HofmannMayboroda}, \cite{HofmannMayborodaMcIntosh}). They also appear if one wants to study maximal regularity operators arising from  some linear or nonlinear partial differential equations (\cite{KT}, \cite{AKMP}). In particular, one wants to  understand  how some (sub)linear operators act on them.  More precisely the following two types of operators appear. First,  
$$\mathcal{T}(F)(x,t):= T_{t}(F(\cdot, t))(x),$$
where $T_{t}$ acts on functions on $\R^n$. Second,  $$\mathcal{T}(F)(x,t): = \int_{0}^\infty T_{t,s}(F(\cdot, s))(x) \frac{ds}{s}$$ where $T_{t,s}$ acts on functions on $\R^n$.  For the second type, we refer to  \cite{AMcR}, \cite{HNP}, \cite{AKMP}.  Positive results on $T^q_{2}$ all rely on the use of $L^2$ off-diagonal estimates (or improved $L^{\min(q,2)}-L^{\max(q,2)}$ off diagonal estimates) and change of angle in the tent space norms.  

For the first type, there is a simple sufficient condition that also depends on the change of angle.
Let us assume that $T_{t}$ acts on $L^2$ functions with compact support and
\begin{equation}
\label{eq:off}
\int_{B(x,t)} |T_{t}(f)(y)|^2\, dy \lesssim 2^{-2j\gamma} \int_{C_{j}(B(x,t))} |f(y)|^2\, dy
\end{equation}
with some $\gamma\ge 0$, provided $f$ is supported in $C_{j}(B(x,t))$ where $C_{j}(B(x,t))=B(x,4t)$ if $j=1$ and 
$C_{j}(B(x,t))=B(x,2^{j+1}t)\setminus B(x,2^jt)$ when $j\ge2$. 
Then 
$$
\mathcal{A}_2(\mathcal{T}(F))(x) \lesssim \sum_{j\ge1} 2^{-j\gamma}\mathcal{A}_2^{(2^{j+1})}(F)(x)
$$
where 
$\mathcal{A}_2^{(\alpha)}$ is defined as $\mathcal{A}_2$ with $B(x,t)$ replaced by $B(x,\alpha t)$ in \eqref{eq:A}. Using the well known change of angle inequality
$$
\| \mathcal{A}_2^{(\alpha)}F\|_{L^q(\R^n)}\lesssim \alpha^{n\max(1/2, 1/q)} \|\mathcal{A}_2F\|_{L^q(\R^n)}
$$
we can conclude for the $T^q_{2}$ boundedness of $\mathcal{T}$ if $\gamma>n\max(1/2, 1/q)$. 
Note in particular that if $\gamma\le n/2$, this argument gives no boundedness, even for $q=2$. Often, the operators  $T_{t}$ are assumed to be uniformly  bounded on $L^2(\R^n)$, which gives $T^2_{2}$ boundedness of  $\mathcal{T}$, whatever $\gamma$. Still, a condition $\gamma>0$ does not seem to guarantee boundedness on $T^q_{2}$ for a range of $q$ about $2$ in general. Thus, there is no available general criterion when $\gamma\le n/2$.

If we let $T_{t}=T$ be independent of $t$ and be a Calder\'on-Zygmund operator, then one obtains 
\eqref{eq:off} with $\gamma=n/2$. Similarly we get $\gamma=n/2$ if we let $T_{t}=\mathcal{M}$ be the centered maximal operator. As said, this argument does not apply. 

On the other hand, it is well-known that if we replace $\mathcal{A}_2$ by the vertical norm
$\mathcal{V}_2$ where $$
\mathcal{V}_r(F)(x)=\left(\int_0^{\infty}|F(x,t)|^r\frac{dt}{t}\right)^{\frac{1}{r}} , \quad x\in \R^n,
$$
then for $T$ being the maximal operator $\mathcal{M}$, 
\begin{equation}
\label{eq:FS}
\|\mathcal{V}_r(\mathcal{M}(F))\|_{L^q(\R^n)} \lesssim \|\mathcal{V}_r(F)\|_{L^q(\R^n)}
\end{equation}
is the vector-valued maximal inequality of Fefferman-Stein, valid when $1<q,r<\infty$ (\cite{FeffermanStein}). It is thus a natural question whether $\mathcal{V}_r$ can be replaced by $\mathcal{A}_r$, that is whether 
the maximal operator, identified with its tensor product with the identity on functions of the $t$ variable, is bounded on $T^q_{r}$.

A modern simple proof of \eqref{eq:FS} is by invoking  extrapolation (see \cite{CruzMartellPerez}): it suffices to prove 
$$
\|\mathcal{V}_r(\mathcal{M}(F))\|_{L^r(w)} \lesssim \|\mathcal{V}_r(F)\|_{L^r(w)}
$$
for any $w\in A_{r}$ to obtain \eqref{eq:FS}, and the latter follows from Muckenhoupt's theorem. 
Thus we are tempted to follow the same route and indeed, we shall prove 
$$
\|\mathcal{A}_r(\mathcal{M}(F))\|_{L^r(w)} \lesssim \|\mathcal{A}_r(F)\|_{L^r(w)}
$$
for any $w\in A_{r}$ using simple upper bounds and known results. We note that the functionals
$\mathcal{V}_r$ and $\mathcal{A}_r$ do not compare on $L^q$ when $q\ne r$,  as shown in   \cite{AuscherHofmannMartell}. Hence,  one cannot deduce such results directly. 

For other operators, we shall also show how extrapolation allows us to conclude tent space boundedness:  we will consider Calder\'on-Zygmund operators,  Riesz potentials and fractional maximal functions, in which case, one looks for $T^p_{r}$ to $T^{q}_{r}$ boundedness for some $q>p$. 
We will also consider singular non integral operators such as the Riesz transform of elliptic operators to test applicability of our methods. In this case, it is a representation of the operator in the form
$\int_{0}^\infty \theta_{s}\, \frac{ds}{s}
$
that is essential. We obtain tent space boundedness with limited range in $q$ and $r$ that is consistent with that of the $L^p$ theory. 

For Calder\'on-Zygmund operators, we shall explore  what happens when $q\le 1$. At $q=1$, we prove a weak-type inequality. We can also take advantage of cancellations in using atomic decompositions at the level of tent spaces.  Then,  atoms need to satisfy the additional condition
$$
\int_{\R^n} A(x,t)\, dx=0, \quad \textrm{for a.e.}\, t>0
$$
and we get results for $q> \frac{n}{n+1}$. Imposing more vanishing moments against polynomials allows to get smaller values of $q$ as it is the case with Hardy spaces on $\R^n$. 

As easy corollaries, we obtain  results in amalgam spaces  in Section  \ref{section:amalgam}.

\section{Main results}\label{sec:main}

As mentioned, if $(T_{t})_{t>0}$ is a family of operators on $\R^n$ acting on (some) measurable functions, we let $\mathcal{T}$ defined by
$$
\mathcal{T}(F)(x,t)=T_{t}(F(\cdot, t))(x), 
$$
provided the formula makes sense, that is provided $F(\cdot, t)$ belongs to an appropriate domain of $T_{t}$. 
If $T$ is a single operator and $T_{t}=T$ for each $t>0$ then $\mathcal{T}= T\otimes I$. In that case and from now on,  we use the same notation by a slight abuse. 

We are ready to state our main results. Precise definitions will be given later.

\begin{theorem}\label{thm:maximal}
Let $\mathcal{M}$ be the centered Hardy-Littlewood maximal operator. For all $1<r<\infty$,
\begin{list}{$(\theenumi)$}{\usecounter{enumi}\leftmargin=1cm \labelwidth=1cm\itemsep=0.2cm\topsep=.0cm \renewcommand{\theenumi}{\alph{enumi}}}

\item 
$
\mathcal{M}:T^{q}_r\rightarrow T^q_r
$, for all $1<q<\infty;$

\item $\mathcal{M}:T^{1}_r\rightarrow wT^1_r$.

\end{list}

\end{theorem}
\begin{theorem}\label{thm:C-Zoperator}
Let $\mathcal{T}$ be a Calder\'on-Zygmund operator on $\R^n$ of order $\delta\in (0,1]$	.  For all $1<r<\infty,
$
\begin{list}{$(\theenumi)$}{\usecounter{enumi}\leftmargin=1cm \labelwidth=1cm\itemsep=0.2cm\topsep=.0cm \renewcommand{\theenumi}{\alph{enumi}}}

\item $
\mathcal{T}:T^q_r\rightarrow T^q_r$, for all $1<q<\infty;$  

\item $
\mathcal{T}:T^1_r\rightarrow wT^1_r;$

\item  $\mathcal{T}:\mathfrak{T}^q_r\rightarrow T_r^q$,\, for all $\frac{n}{n+\delta}<q\leq 1;$

\item  $\mathcal{T}:\mathfrak{T}^q_r\rightarrow \mathfrak{T}^q_r$,\, for all $\frac{n}{n+\delta}<q\leq 1$, if $\mathcal{T}^*(1)=0$.

\end{list}
\end{theorem}

\begin{theorem}\label{thm:rieszpotential}
For $0<\alpha<n$, $\frac{n}{n-\alpha}<r<\infty$, and $1<p<q<\infty$ such that $\frac{1}{p}-\frac{1}{q}=\frac{\alpha}{n}$, 
$$
\mathcal{I}_{\alpha}, \mathcal{M}_{\alpha}:T^p_{r}\rightarrow T^q_{r}.
$$
\end{theorem}%
\begin{theorem}\label{thm:riesztransform} Let $L=-\div(A\nabla)$ be an elliptic operator with complex-valued coefficients. 
For $q_-(L)<q,r<q_+(L)$ we have
$$
\nabla L^{-\frac{1}{2}}:T^q_r\rightarrow T^q_r.
$$
\end{theorem}

Here is an interesting corollary. 

\begin{corollary}\label{co:maximal}
Assume $(T_{t})_{t>0}$ is a family of operators with $\sup_{t>0}|T_{t}(f)|\le (\mathcal{M}|f|^\rho)^{1/\rho}$ for some $\rho\ge 1$. For all $\rho<q,r<\infty$, 
\begin{align}\label{maximal:2}
\mathcal{T}:T^{q}_r\rightarrow T^q_r.
\end{align}
\end{corollary}
  
This applies to the heat semigroup $e^{t^2\Delta}$ or the Poisson semigroup $e^{-t\sqrt {-\Delta}}$. 
Note that, in both cases, there is enough decay. Often, the sup norm is too strong an hypothesis. Here is a weaker one, applying for example to semigroups $e^{-t^2L}$ associated to elliptic operators such as the ones in Section \ref{sec:RT}.

\begin{corollary}\label{co:maximal2}
Assume $(T_{t})_{t>0}$ is a family of operators with a kind of reverse H\"older estimate
$$\left(\dashint_{B(x,t)}|T_{t}(f)(y)|^s dy\right)^{\frac{1}{s}}\le\left(\dashint_{B(x,\alpha t)}|\mathcal{M} (|f|^\rho)(y)|\ dy\right)^{\frac{1}{\rho}} $$ for some $\alpha>1$ and  some $1 \le \rho < s$, uniformly for all $(x,t)\in \R^{n+1}_{+}$. Then, for all $(r,q)$ with $\rho<r \le s$ and $\rho<q<\infty,$
\begin{align}\label{maximal:3}
\mathcal{T}:T^{q}_r\rightarrow T^q_r.
\end{align}
\end{corollary}

This follows from the pointwise inequality
$$
\mathcal{A}_r(\mathcal{T}(F))(x) \le \left(\mathcal{A}^{(\alpha)}_{\frac{r}{\rho}}(\mathcal{M}(|F|^\rho))(x)\right)^{\frac{1}{\rho}}
$$
 with $\mathcal{T}(F)(x,t)= T_{t}(F(\cdot, t))(x)$, and Theorem \ref{thm:maximal}.

\section{Weights}

Since we are going to use some weight theory, let us recall  definitions and some properties. We say that a  function, $w$, is a  weight if  $w\in L^1_{loc}$ and $w(x)>0$ for a.e. $x\in \R^n$. For $1<p<\infty$ if $B$ represents a ball in $\R^n$ we say that $w\in A_p$ if
$$
\left(\dashint_Bw(x)dx\right)\left(\dashint_Bw(x)^{1-p'}dx\right)^{p-1}\leq C, \,\textrm{for all } B\subset \R^n.
$$
For $p=1$, $w\in A_1$ if 
$$
\dashint_Bw(y)dy\leq C w(x), \,\textrm{for }a.e.\,x\in B\,\textrm{and for all }B\subset\R^n.
$$
We introduce also the reverse H\"older classes. For $1<q<\infty$ we say that $w\in RH_{q}$ if 
$$
\left(\dashint_Bw(x)^{q}dx\right)^{\frac{1}{q}}\leq C\dashint_Bw(x)dx, \,\textrm{for all } B\subset \R^n.
$$
And for $q=\infty$, $w\in RH_{\infty}$ if
$$
w(x)\leq C \dashint_Bw(y)dy, \,\textrm{for }a.e.\,x\in B\,\textrm{and for all }B\subset\R^n.
$$

We sum up some of the properties of these classes in the following result, see for instance \cite{GCRF85},
\cite{Duo}, or \cite{Grafakos}.

\begin{proposition}\label{prop:weights}\
\begin{enumerate}
\renewcommand{\theenumi}{\roman{enumi}}
\renewcommand{\labelenumi}{$(\theenumi)$}
\addtolength{\itemsep}{0.2cm}

\item $A_1\subset A_p\subset A_q$ for $1\le p\le q<\infty$.

\item $RH_{\infty}\subset RH_q\subset RH_p$ for $1<p\le q\le \infty$.

\item If $w\in A_p$, $1<p<\infty$, then there exists $1<q<p$ such
that $w\in A_q$.

\item If $w\in RH_s$, $1<s<\infty$, then there exists $s<r<\infty$ such
that $w\in RH_r$.

\item $\displaystyle A_\infty=\bigcup_{1\le p<\infty} A_p=\bigcup_{1<s\le
\infty} RH_s$.

\item If $1<p<\infty$, $w\in A_p$ if and only if $w^{1-p'}\in
A_{p'}$.

\end{enumerate}
\end{proposition}
\section{Hardy Littlewood maximal operator}
The centered Hardy-Littlewood maximal operator is defined for locally integrable $f$ by
\begin{align*}
\mathcal{M}(f)(x)=\sup_{\tau>0}\dashint_{B(x,\tau)}|f(y)|\, dy.
\end{align*}
For this operator we use the following pointwise inequality.
\begin{lemma}\label{lemma:H-Lmaximal}
For all $x\in \R^n$, $t>0$, and $1<r<\infty$, and all $f$ locally $r$ integrable,  we have that 
\begin{align}\label{pointwisemaximal}
\left(\dashint_{B(x,t)}|\mathcal{M}(f)(y)|^r\, dy\right)^{\frac{1}{r}}
\lesssim
\left(\dashint_{B(x,2t)}|f(y)|^r\, dy\right)^{\frac{1}{r}}
+
\mathcal{M}_u\left(\dashint_{B(\cdot,t)}|f(z)|dz\right)(x),
\end{align}
where $\mathcal{M}_u$ is the uncentered maximal operator.
\end{lemma}
\begin{proof}
Fix $x\in \R^n$ and $t>0$, and split the  supremum into $0<\tau\leq t$ and $t<\tau$. Then,
\begin{multline*}
\left(\dashint_{B(x,t)}|\mathcal{M}(f)(y)|^r\, dy\right)^{\frac{1}{r}}
\leq 
\left(\dashint_{B(x,t)}\left(\sup_{0<\tau\leq t}\dashint_{B(y,\tau)}|f(z)|dz\right)^r\, dy\right)^{\frac{1}{r}}
\\
+
\left(\dashint_{B(x,t)}\left(\sup_{\tau> t}\dashint_{B(y,\tau)}|f(z)|dz\right)^r\, dy\right)^{\frac{1}{r}}=:I+II.
\end{multline*}
Now, since, for $0<\tau\leq t$ and $y\in B(x,t)$ it happens that $B(y,\tau)\subset B(x,2t)$,
\begin{multline*}
I
\leq 
\left(\dashint_{B(x,t)}\left(\sup_{0<\tau\leq t}\dashint_{B(y,\tau)}|f(z)\chi_{B(x,2t)}(z)|dz\right)^r\, dy\right)^{\frac{1}{r}}
\\
\leq 
\left(\dashint_{B(x,t)}|\mathcal{M}(f\chi_{B(x,2t)})(y)|^r\, dy\right)^{\frac{1}{r}}
\lesssim
\left(\dashint_{B(x,2t)}|f(y)|^r\, dy\right)^{\frac{1}{r}},
\end{multline*}
where in the last inequality we have used that $\mathcal{M}:L^r(\R^n)\rightarrow L^r(\R^n)$ (\cite[Theorem 2.5]{Duo}).

As for $II$, note that, for $\xi, z\in \R^n$,  $\xi\in B(z,t)\Leftrightarrow z\in B(\xi,t)$, and also that if $z\in B(y,\tau)$, $\xi\in B(z,t)$, and $\tau>t$, then $\xi\in B(y,2\tau)$. Besides, observe that the fact that $x\in B(y,t)$ and $\tau>t$ implies that $x\in B(y,2\tau)$. Hence, applying Fubini's theorem,
\begin{multline*}
II
=
\left(\dashint_{B(x,t)}\left(\sup_{\tau> t}\dashint_{B(y,\tau)}|f(y)|\dashint_{B(z,t)}d\xi\,dz\right)^r\, dy\right)^{\frac{1}{r}}
\\
\leq 
\left(\dashint_{B(x,t)}\left(\sup_{\tau> t}\dashint_{B(y,2\tau)}\dashint_{B(\xi,t)}|f(z)|dz\,d\xi\right)^r\, dy\right)^{\frac{1}{r}}
\\
\lesssim
\mathcal{M}_u\left(\dashint_{B(\cdot,t)}|f(z)|dz\right)(x).
\end{multline*}
Gathering the estimates obtained for $I$ and $II$, we conclude 
\eqref{pointwisemaximal}.
\end{proof}

\subsection{Proof of Theorem \ref{thm:maximal}, part $(a)$}
Let $w$ be a Muckenhoupt weight.
We shall prove for all $w\in A_r$ that  for all $F\in T^r_{r}$ (hence $F(\cdot, t)$ is locally $r$ integrable for almost every $t>0$)
\begin{align}\label{maximal:tentspace2}
\int_{\R^n}|\mathcal{A}_r(\mathcal{M}(F))(x)|^rw(x)dx
\leq 
C \int_{\R^n}|\mathcal{A}_r(F)(x)|^rw(x)dx.
\end{align}
From this,  by \cite[Theorem 3.9]{CruzMartellPerez}, we have that, for all $1<q<\infty$, $F\in T^r_{r}$ and $w_0\in A_q$,
\begin{align*}
\int_{\R^n}|\mathcal{A}_r(\mathcal{M}(F))(x)|^qw_0(x)dx
\leq 
C \int_{\R^n}|\mathcal{A}_r(F)(x)|^qw_0(x)dx.
\end{align*}
In particular, for $w_0\equiv 1$, we have that $w_0\in A_q$ for all $1<q<\infty$, then, for all $F\in T^r_{r}$,
\begin{align}\label{maximaldensity}
\|\mathcal{M}F\|_{T^q_r}=\left(\int_{\R^n}|\mathcal{A}_r(\mathcal{M}(F))(x)|^qdx\right)^{\frac{1}{q}}
\leq 
C \left(\int_{\R^n}|\mathcal{A}_r(F)(x)|^qdx\right)^{\frac{1}{q}}=
C\|F\|_{T^q_r}.
\end{align}
Approximating $ T^q_{r}$ functions by compactly supported $T^r_{r}$ functions,  we conclude \eqref{maximaldensity} holds for functions $F\in T^q_r$ by the monotone convergence theorem.

Therefore, to finish the proof it just remains to show \eqref{maximal:tentspace2}. This follows by \eqref{pointwisemaximal}  applied to $f=F(\cdot, t)$
and the fact that $\mathcal{M}_u:L^r(w)\rightarrow L^r(w)$, for all $w\in A_r$ (\cite{Muckenhoupt}). Using these two facts and \cite[Proposition 3.2]{MartellPrisuelos}, for all $w\in A_r$,
\begin{multline*}
\int_{\R^n}|\mathcal{A}_r(\mathcal{M}(F))(x)|^rw(x)dx
=
\int_{\R^n}\int_0^{\infty}\dashint_{B(x,t)}|\mathcal{M}(F(\cdot,t))(y)|^r\frac{dy\,dt}{t}\,w(x)dx
\\
\lesssim
\int_{\R^n}\int_0^{\infty}\dashint_{B(x,2t)}|F(y,t)|^r\frac{dy\,dt}{t}\,w(x)dx
+
\int_{\R^n}\int_0^{\infty}\left|\mathcal{M}_{u}\left(\dashint_{B(\cdot,t)}|F(y,t)|dy\right)(x)\right|^r\frac{dt}{t}\,w(x)dx
\\
\lesssim
\int_{\R^n}|\mathcal{A}_r(F)(x)|^r\,w(x)dx
+
\int_0^{\infty}\int_{\R^n}\left|\mathcal{M}_{u}\left(\dashint_{B(\cdot,t)}|F(y,t)|dy\right)(x)\right|^rw(x)dx\frac{dt}{t}\,
\\
\lesssim
\int_{\R^n}|\mathcal{A}_r(F)(x)|^r\,w(x)dx
+
\int_0^{\infty}\int_{\R^n}\left(\dashint_{B(x,t)}|F(y,t)|dy\right)^rw(x)dx\frac{dt}{t}\,
\\
\lesssim
\int_{\R^n}|\mathcal{A}_r(F)(x)|^r\,w(x)dx.
\end{multline*}
\qed
\subsection{Proof of Theorem \ref{thm:maximal}, part $(b)$}
By \eqref{pointwisemaximal} and the change of angle in tent spaces, for all $\lambda>0$, we have that
\begin{align*}
\lambda|\{x\in \R^n:\mathcal{A}_r(\mathcal{M}(F))(x)>\lambda\}|\lesssim \|F\|_{T^1_r}+\lambda\left|\left\{x\in \R^n:\mathcal{V}_r(\mathcal{M}_u(\widetilde{F}))(x)>\frac{\lambda}{2}\right\}\right|,
\end{align*}
where $\widetilde{F}(x,t):=\dashint_{B(x,t)}|F(z,t)|\,dz$. Then, applying  the Fefferman-Stein vector-valued weak type $(1,1)$ inequality \cite{FeffermanStein}, we control the second term in the above sum  by
$$
C\int_{\R^n}\left(\int_{0}^{\infty}|\widetilde{F}(x,t)|^r\frac{dt}{t}\right)^{\frac{1}{r}}dx\lesssim \|F\|_{T^1_r},
$$
for some constant $C>0$. Therefore, taking the supremum over all $\lambda>0$, conclude that
$$
\|\mathcal{M}F\|_{wT^1_r}\lesssim \|F\|_{T^1_r}.
$$
\qed
\section{Calder\'on-Zygmund operators}\label{section:C-Z}
Recall that $\mathcal{T}$ is a Calder\'on-Zygmund operator of order $\delta \in (0,1]$ if 
$\mathcal{T}$ is bounded on $L^2(\R^n)$ and has a kernel representation  
\begin{align*}
\mathcal{T}(f)(x)=\int_{\R^n}K(x,y)f(y)\,dy,
\end{align*}
for almost every $x$ not in the support of $f\in L^2(\R^n)$,
with the kernel, $K$, satisfying the standard conditions: for some $\delta>0$,
\begin{align}\label{kernel1}
|K(x,y)|\leq \frac{C}{|x-y|^n}, \textrm{ for } x\ne y;
\end{align}
\begin{align}\label{kernel2}
|K(x,y)-K(x,z)|\leq C\frac{|y-z|^{\delta}}{|x-y|^{n+\delta}},
 \textrm{ for } |x-y|>2|y-z|;
\end{align}
\begin{align}\label{kernel3}
|K(x,y)-K(w,y)|\leq C \frac{|x-w|^{\delta}}{|x-y|^{n+\delta}}, \textrm{ for } |x-y|>2|x-w|.
\end{align}
Classically, $\mathcal{T}$ extends to a bounded operator on $L^r(\R^n)$ for $1<r<\infty$  (see for instance \cite[Theorem 5.10]{Duo}) and the kernel representation holds also when $f\in L^r(\R^n)$.
The following lemma gives us a useful pointwise inequality for Calder\'on-Zygmund operators. 
\begin{lemma}\label{lemma:C_Z}
Let $\mathcal{T}$ be a Calder\'on-Zygmund operator and   $f\in L^r(\R^n)$. We have, for $1<r<\infty$, and for each $x\in \R^n$ and all $t>0$,
\begin{align}\label{C-Z:pointwise}
\left(\dashint_{B(x,t)}|\mathcal{T}(f)(y)|^rdy\right)^{\frac{1}{r}}
\lesssim 
\left(\dashint_{B(x,2t)}|f(y)|^rdy\right)^{\frac{1}{r}}
+
\mathcal{T}_{*}(f)(x)
+\mathcal{M}(f)(x),
\end{align}
where $\mathcal{T}_{*}(f)(x):=\sup_{\varepsilon>0}\left|\int_{|x-y|>\varepsilon}K(x,y)f(y)dy\right|.$
\end{lemma}
\begin{proof}
Fix $x\in \R^n$ and $t>0$, consider the ball $B(x,2t)$ and write  $f=f\chi_{B(x,2t)}+f\chi_{\R^n\setminus B(x,2t)}=:f_{loc}+f_{glob}$. Then
\begin{equation*}
 \left(\dashint_{B(x,t)}|\mathcal{T}(f)(y)|^rdy\right)^{\frac{1}{r}}
 \leq
  \left(\dashint_{B(x,t)}|\mathcal{T}(f_{loc})(y)|^rdy\right)^{\frac{1}{r}}
 +
 \left(\dashint_{B(x,t)}|\mathcal{T}(f_{glob})(y)|^rdy\right)^{\frac{1}{r}}
 =:I+II.
\end{equation*}
Since  $\mathcal{T}:L^r(\R^n)\rightarrow L^r(\R^n)$, 
$$
I\lesssim \left(\dashint_{B(x,2t)}|f(y)|^rdy\right)^{\frac{1}{r}}.
$$
As for $II$,  apply the fact that, for $y\in B(x,t)$, $\{z:|x-z|>2t\}\subset\{z:
|x-z|>2|x-y|\}$ and  \eqref{kernel3}. Then,
\begin{align*}
II&
=
\left(\dashint_{B(x,t)}\left|\int_{\R^n}K(y,z)f_{glob}(z)\,dz\right|^r\,dy\right)^{\frac{1}{r}}
=
\left(\dashint_{B(x,t)}\left|\int_{|x-z|>2t}K(y,z)f(z)\,dz\right|^r\,dy\right)^{\frac{1}{r}}
\\
&
= 
\left(\dashint_{B(x,t)}\left|\int_{|x-z|>2t}(K(y,z)-K(x,z)+K(x,z))f(z)\,dz\right|^rdy\right)^{\frac{1}{r}}
\\
&\leq
 \left(\dashint_{B(x,t)}\left(\int_{|x-z|>2t}|K(y,z)-K(x,z)|\, |f(z)|\,dz\right)^rdy\right)^{\frac{1}{r}}
\\
& \quad 
 +
 \left(\dashint_{B(x,t)}\left|\int_{|x-z|>2t}K(x,z)f(z)\,dz\right|^rdy\right)^{\frac{1}{r}}
 \\
&\lesssim
 \left(\dashint_{B(x,t)}\left(\int_{|x-z|>2t}\frac{|x-y|^{\delta}}{|x-z|^{n+\delta}}\, |f(z)|\,dz\right)^rdy\right)^{\frac{1}{r}}
 +
\left|\int_{|x-z|>2t}K(x,z)f(z)\,dz\right|
 \\
&\lesssim
 \left(\dashint_{B(x,t)}\left(\sum_{k=0}^{\infty}\int_{2^{k+1}t<|x-z|\leq 2^{k+2}t}\frac{|x-y|^{\delta}}{|x-z|^{n+\delta}}\, |f(z)|\,dz\right)^rdy\right)^{\frac{1}{r}}
+\left|\int_{|x-z|>2t}K(x,z)f(z)\,dz\right|
\\
&\lesssim \sum_{k=0}^{\infty}\frac{1}{2^{k\delta}}
 \left(\dashint_{B(x,t)}\left(\dashint_{B(x,2^{k+2}t)}|f(z)|\,dz\right)^rdy\right)^{\frac{1}{r}}
+\left|\int_{|x-z|>2t}K(x,z)f(z)\,dz\right|
\\
&\lesssim \sum_{k=0}^{\infty}\frac{1}{2^{k\delta}}
 \dashint_{B(x,2^{k+2}t)}|f(z)|\,dz
+\left|\int_{|x-z|>2t}K(x,z)f(z)\,dz\right|
\\
&\lesssim \mathcal{M}(f)(x)
+\mathcal{T}_*(f)(x).
\end{align*}
\end{proof}

\subsection{Proof of Theorem \ref{thm:C-Zoperator}, part $(a)$}
As we said above we first use \eqref{C-Z:pointwise} to prove a weighted version of the case $q=r$ for $\mathcal{T}$. We recall that we use the same notation $\mathcal{T}$ for its extension to tent spaces.

We consider $F\in T^r_{r}$   so that for almost every $t>0$, $F(\cdot, t)\in L^r(\R^n)$ and all calculations make sense. For a weight $w\in A_r\cap RH_{\infty}$, by \eqref{C-Z:pointwise}, Fubini, the fact that $\mathcal{T}^*, \mathcal{M}:L^r(w)\rightarrow L^r(w)$ (see for instance \cite{CoifmanFefferman}, \cite[Theorem 7.13]{Duo}), and applying \cite[Proof of Proposition 2.3]{AuscherHofmannMartell}, and \cite[Proposition 3.2]{MartellPrisuelos}, 
\begin{align*}
\Bigg(\int_{\R^n}\int_{0}^{\infty}\int_{B(x,t)}&|\mathcal{T}(F(\cdot,t))(y)|^r\frac{dy\,dt}{t^{n+1}}w(x)\,dx\Bigg)^{\frac{1}{r}}
\\
&
\lesssim
\Bigg(\int_{\R^n}\int_{0}^{\infty}\int_{B(x,2t)}|F(y,t)|^r\frac{dy\,dt}{t^{n+1}}w(x)\,dx\Bigg)^{\frac{1}{r}}
\\
&\quad+
\left(\int_{\R^n}\int_{0}^{\infty}|\mathcal{T}^*(F(\cdot,t))(x)|^r\frac{dt}{t}w(x)\,dx\right)^{\frac{1}{r}}
+
\left(\int_{\R^n}\int_{0}^{\infty}|\mathcal{M}(F(\cdot,t))(x)|^r\frac{dt}{t}w(x)\,dx\right)^{\frac{1}{r}}
\\
&\lesssim
\left(\int_{\R^n}\int_{0}^{\infty}\int_{B(x,t)}|F(y,t)|^r\frac{dy\,dt}{t^{n+1}}w(x)\,dx\right)^{\frac{1}{r}}
+
\left(\int_{\R^n}\int_{0}^{\infty}|F(x,t)|^r\frac{dt}{t}w(x)\,dx\right)^{\frac{1}{2}}
\\
&\lesssim
\left(\int_{\R^n}\int_{0}^{\infty}\int_{B(x,t)}|F(y,t)|^r\frac{dy\,dt}{t^{n+1}}w(x)\,dx\right)^{\frac{1}{r}}.
\end{align*}
Therefore, for all $w\in A_r\cap RH_{\infty}$ and $F\in T^r_{r}$, 
\begin{align}\label{extra:C-Z}
\int_{\R^n}\int_{0}^{\infty}\int_{B(x,t)}|\mathcal{T}(F(\cdot,t))(y)|^r\frac{dy\,dt}{t^{n+1}}w(x)\,dx
\lesssim
\int_{\R^n}\int_{0}^{\infty}\int_{B(x,t)}|F(y,t)|^r\frac{dy\,dt}{t^{n+1}}w(x)\,dx.
\end{align}
In particular for $w\equiv 1$ and $F$ as above, 
$$
\|\mathcal{T}(F)\|_{T_r^r}\lesssim \|F\|_{T^r_r},
$$
where the estimate does not depend on $F$. This proves the case $q=r$. 
Note now that in view of \eqref{extra:C-Z}, we can apply \cite[Theorem 3.31]{CruzMartellPerez}, for $p_-=1$ and $p_+=r$. Then, we obtain that, for all
$1<q<r$ and $w_0\in A_q\cap RH_{\left(\frac{r}{q}\right)'}$, and all $F\in T^r_{r}$, 
\begin{align*}
\int_{\R^n}\left(\int_{0}^{\infty}\int_{B(x,t)}|\mathcal{T}(F(\cdot,t))(y)|^r\frac{dy\,dt}{t^{n+1}}\right)^{\frac{q}{r}}w_0(x)\,dx
\lesssim
\int_{\R^n}\left(\int_{0}^{\infty}\int_{B(x,t)}|F(y,t)|^r\frac{dy\,dt}{t^{n+1}}\right)^{\frac{q}{r}}w_0(x)\,dx.
\end{align*}
Hence, taking $w_0\equiv 1$, we have in particular that $w_0\in A_q\cap RH_{\left(\frac{r}{q}\right)'}$. Then, for $1<q<r$ and all $F\in T^r_{r}$, 
$$
\|\mathcal{T}(F)\|_{T_r^q}\lesssim \|F\|_{T^q_r}.
$$
We conclude by density of $T^r_{r}\cap T^q_{r}$ into $T^q_{r}$. 

 In order to prove the boundedness for $1<r<q<\infty$, we use a duality argument. Take $F\in T_r^q\cap T^r_{r}$ and $G\in T^{q'}_{r'}\cap T^{r'}_{r'}$.  
By the previous argument and dualization we obtain, 
$$
\int_{\R^n}\int_0^{\infty}|F(y,t)\widetilde{\mathcal{T}}(G(\cdot,t))(y)|\frac{dt\,dy}{t}
\lesssim \|F\|_{T^q_r}\|G\|_{T^{q'}_{r'}},
$$
where $\widetilde{\mathcal{T}}$ is the adjoint of $\mathcal{T}$.
Also
$$
\int_{\R^n}\int_0^{\infty}\left|\mathcal{T}(F(\cdot,t))(x)G(x,t)\right|\frac{dt\,dx}{t} \lesssim  \|F\|_{T^r_r}\|G\|_{T^{r'}_{r'}} <\infty$$
 Thus,   Fubini's theorem 
and
$$
\int_{\R^n}\mathcal{T}(F(\cdot,t))(x)G(x,t)\, dx= \int_{\R^n}F(y,t)\widetilde{\mathcal{T}}(G(\cdot,t))(y)\,dy
$$
yield
$$
\left|\int_{\R^n}\int_0^{\infty}\mathcal{T}(F(\cdot,t))(x)G(x,t)\frac{dt\,dx}{t}\right|
=
\left|\int_{\R^n}\int_0^{\infty}F(y,t)\widetilde{\mathcal{T}}(G(\cdot,t))(y)\frac{dt\,dy}{t}\right|
\lesssim \|F\|_{T^q_r}\|G\|_{T^{q'}_{r'}}.
$$
Finally, taking the supremum over all $G$ as above, such that $\|G\|_{T^{q'}_{r'}}\leq 1$,  we conclude that, for all $F \in T^q_{r}\cap T^r_{r}$,
$
\|\mathcal{T}(F)\|_{T^q_r}\lesssim \|F\|_{T^q_r}.
$
By density, this allows to extend the action of $\mathcal{T}$ to all $F\in T^q_{r}$.
\qed

Remark that 
$$
\int_{\R^n}\int_0^{\infty}\left|\mathcal{T}(F(\cdot,t))(x)G(x,t)\right|\frac{dt\,dx}{t}<\infty
$$
 for all $F\in T^q_{r}$ and all $G\in T^{q'}_{{r'}}$ when $q=r$. But when $q\ne r$, the argument does not allow to conclude for the convergence of this integral for arbitrary  $F\in T^q_{r}$ and $G\in T^{q'}_{{r'}}$. Of course, this inequality holds for the extension of $\mathcal{T}$ on $T^q_{r}$.
\subsection{Proof of Theorem \ref{thm:C-Zoperator}, part $(b)$} Let $F\in T^r_{r} \cap T^1_{r}$, which is dense in $T^1_{r}$.  It follows from \eqref{C-Z:pointwise} that 
$$
\mathcal{A}_r(\mathcal{T}(F)) \lesssim \mathcal{A}_r^{(2)}(F) + \mathcal{V}_r(\mathcal{M}(F)) + \mathcal{V}_r(\mathcal{T}_{*}(F)).
$$
We need to estimate the $L^{1,\infty}(\R^n)$ norm of each term.

The first  term  has $L^1(\R^n)$ norm  controlled  by $c\|F\|_{T^1_r}$ for some constant $c>0$ by change of angle in tent spaces. 

For the second one, one applies Fefferman-Stein  vector-valued weak type $(1,1)$ inequality and then, the fact that the norm in $L^1(\R^n)$ of the vertical function $\mathcal{V}_{r}(F)$ is controlled by the norm in $L^1(\R^n)$ of the conical function $\mathcal{A}_r(F)$ (see \cite{AuscherHofmannMartell}).

For the third term, the needed weak type estimate is    
$$
\left|\left\{x\in \R^n:\mathcal{V}_r(\mathcal{T}_{*}(F))(x)>{\lambda}\right\}\right|\lesssim \frac{\|\mathcal{V}_r(F)\|_{L^1(\R^n)}}{\lambda}.
$$
It should be known but as we have not been able to locate a proof, we provide one for the reader's comfort. Once this is proved, we use again the  result in \cite{AuscherHofmannMartell} mentioned above. 

Fix $\lambda>0$ and consider the set
$$
\Omega_{\lambda}:=\{x\in \R^n:\mathcal{M}_u(\mathcal{V}_r(F))(x)>\lambda\},
$$ 
where we  recall that $\mathcal{M}_u$ represents the uncentered Hardy-Littlewood maximal operator. We have that $\Omega_\lambda$ is open and, since $\|\mathcal{V}_r(F)\|_{L^1(\R^n)}<\infty$, we conclude that $|\Omega_{\lambda}|<\infty$. Therefore, we can take a Whitney decomposition $\Omega_{\lambda}=\bigcup_{i\in \N}Q_i$, where $Q_i$ are dyadic and disjoint cubes such that
$$ 
\sqrt{n}\ \ell(Q_i)\le \textrm{dist}(Q_i,\R^n\setminus \Omega_{\lambda})<4
\sqrt{n}\ \ell(Q_i).
$$
Hence,
$$
\mathcal{V}_r(F)(x)\leq \lambda, \, \textrm{ for a.e.}\, x\in \R^{n}\setminus \Omega_{\lambda} , \quad \dashint_{Q_i}|\mathcal{V}_r(F)(x)|\,dx\leq 8^n\lambda, \quad \textrm{and}\quad |\Omega_{\lambda}|\leq\frac{1}{\lambda}\int_{\Omega_\lambda}|\mathcal{V}_r(F)(x)|\,dx.
$$
Then if we set 
$$
\mathcal{G}=\mathcal{V}_r(F)\chi_{\R^n\setminus \Omega_{\lambda}}+\sum_{i\in \N}\left(\dashint_{Q_i}\mathcal{V}_r(F)\right)\chi_{Q_i}, \,\textrm{ and }\,
\mathcal{B}=\sum_{i\in \N}\left(\mathcal{V}_r(F)-\dashint_{Q_i}\mathcal{V}_r(F)\right)\chi_{Q_i}
$$
we have that $\mathcal{V}_r(F)=\mathcal{G}+\mathcal{B}$ is a Calder\'on-Zygmund decomposition of $\mathcal{V}_r(F)$ at heigh $\lambda$ satisfiying:
$$
|\mathcal{G}(x)|\leq 10^n\lambda,\,\textrm{for a.e. }x\in \R^n,\quad \|\mathcal{G}\|^r_{L^{r}(\R^n)}\leq (10^n\lambda)^{r-1}\|\mathcal{V}_r(F)\|_{L^1(\R^n)}, 
$$
$$
\int_{Q_i}\mathcal{B}(x)dx=0,\,\dashint_{Q_i}|\mathcal{B}(x)|dx\leq 2
\dashint_{Q_i}|\mathcal{V}_r(F)(x)|dx,\,\textrm{ and }\, \|\mathcal{B}\|_{L^1(\R^n)}\leq 2\|\mathcal{V}_r(F)\|_{L^1(\R^n)}.
$$
Now set $F=G+H,$
where 
$$
G(x,t)=F(x,t)\chi_{\R^{n}\setminus\Omega_{\lambda}}(x)+\sum_{i\in \N}\chi_{Q_i}(x)\dashint_{Q_i}F(y,t)\, dy,$$
and
$$
H(x,t)=\sum_{i\in \N}\chi_{Q_i}(x)\left(F(x,t)-\dashint_{Q_i}F(y,t)\, dy\right)=:\sum_{i\in \N}H_i(x,t).
$$
Then,
\begin{multline*}
\left|\left\{x\in \R^n:\mathcal{V}_r(\mathcal{T}_{*}(F))(x)>{\lambda}\right\}\right|
\\
\leq
\left|\left\{x\in \R^n:\mathcal{V}_r(\mathcal{T}_{*}(G))(x)>\frac{\lambda}{2}\right\}\right|
+
\left|\left\{x\in \R^n:\mathcal{V}_r(\mathcal{T}_{*}(H))(x)>\frac{\lambda}{2}\right\}\right|=:I+II.
\end{multline*}
Applying Chebychev's inequality and the $ L^r(\R^n)$ boundedness of $\mathcal{T}_{*}$, we obtain
\begin{multline*}
I
\lesssim \frac{1}{\lambda^r}\int_{\R^n}\mathcal{V}_r(\mathcal{T}_{*}(G))(x)^r\,dx
\\
=
\frac{1}{\lambda^r}\int_0^{\infty}\int_{\R^n}|\mathcal{T}_{*}(G(\cdot, t))(x)|^r\,dx\frac{dt}{t}
\lesssim
\frac{1}{\lambda^r}\int_0^{\infty}\int_{\R^n}|G(x,t)|^r\,dx\frac{dt}{t}
\\
=\frac{1}{\lambda^r}\int_{\R^n}\int_0^{\infty}\left|F(x,t)\chi_{\R^n\setminus \Omega_{\lambda}}(x)+\sum_{i\in \N}\chi_{Q_i}(x)\dashint_{Q_i}|F(y,t)|\,dy
\right|^r\frac{dt}{t}\,dx
\\
\lesssim
\frac{1}{\lambda^r}\int_{\R^n}\left|\chi_{\R^n\setminus \Omega_{\lambda}}(x)\mathcal{V}_r(F)(x)+\sum_{i\in \N}\chi_{Q_i}(x)\dashint_{Q_i}\mathcal{V}_r(F)(y)\,dy
\right|^r\,dx
\\
= \frac{1}{\lambda^r}
\|\mathcal{G}\|^r_{L^r(\R^n)}
\lesssim 
\frac{1}{\lambda}\|\mathcal{V}_r(F)\|_{L^1(\R^n)}.
\end{multline*}
As for the estimate of $II$, note that
\begin{equation*}
II
\lesssim
\left|\bigcup_{i\in \N}2\sqrt n\, Q_i\right|+
\left|\left\{x\in \R^n\setminus\left(\bigcup_{i\in \N}2\sqrt n\, Q_i\right) :\mathcal{V}_r(\mathcal{T}_{*}(H))(x)>\frac{\lambda}{2}\right\}\right|.
\end{equation*}
Then, since
$$
\left|\bigcup_{i\in \N}2\sqrt n\, Q_i\right|
\lesssim
|\Omega_{\lambda}|
\lesssim
\frac{1}{\lambda}\|\mathcal{V}_r(F)\|_{L^1(\R^n)},
$$
we just need to consider the second term in the previous sum.  For  $t>0$, and $x\in \R^n\setminus\left(\bigcup_{j\in \N}2\sqrt n\, Q_j\right)$, let us study the $\mathcal{T}_{*}(H)(x,t)$. Pick $\varepsilon>0$ and consider 
$$
\left|\int_{|x-y|>\varepsilon}
K(x,y)H(y,t)\,dy\right|= \left| \sum_{i\in \N}\int_{|x-y|>\varepsilon}
K(x,y)H_{i}(y,t)\,dy\right|.
$$
We distinguish three possible cases in the series. Case $1$: $Q_i\subset B(x,\varepsilon)$. Then,
$Q_i\cap(\R^n\setminus B(x,\varepsilon))=\emptyset$, and consequently
$$
\left|\int_{|x-y|>\varepsilon}
K(x,y)H_i(y,t)\,dy\right|=0.
$$ 
Case $2$: $Q_i\subset \R^n\setminus B(x,\varepsilon)$. Call $y_{i}$ the centre of $Q_{i}$ and $\ell(Q_{i})$ its length. Since $x\in \R^n\setminus 2\sqrt n\, Q_i$,  we have   $|x-y_i|>2|y-y_i|$ for any $y\in Q_{i}$. As   $\supp(H_i)\subset Q_i\subset \R^n\setminus B(x,\varepsilon)$, we can use the mean value $\int_{\R^n}H_i(y,t)\, dy=0$ to obtain 
 \begin{align*}
\left|\int_{|x-y|>\varepsilon}
K(x,y)H_i(y,t)\,dy\right|
& \leq
\int_{Q_{i}}
|K(x,y)-K(x,y_i)||H_i(y,t)|\,dy
\\ &
\lesssim 
\frac{\ell(Q_i)^{\delta}}{|x-y_i|^{n+\delta}}\int_{Q_i}
|H_i(y,t)|\,dy \lesssim \frac{\ell(Q_i)^{\delta}}{|x-y_i|^{n+\delta}}\int_{Q_i}
|F(y,t)|\,dy.
\end{align*}
Case $3$: $B(x,\varepsilon)\cap Q_i\neq \emptyset$ but $Q_i\nsubset B(x,\varepsilon)$. Note that  $\varepsilon> \sqrt{n}\, \ell(Q_i)/2$. Indeed, if 
$x_0\in B(x,\varepsilon)\cap Q_i$, $\sqrt{n}\, \ell(Q_i)\le |x-y_i|\leq |x_0-x|+|x_0-y_i|< \varepsilon+\frac{\sqrt{n}\, \ell(Q_i)}{2}$ hence $ \sqrt{n}\, \ell(Q_i)/2<\varepsilon$. It follows that $Q_i\subset B(x,3\varepsilon)$. Hence, 
\begin{multline*}
\left|\int_{|x-y|>\varepsilon}
K(x,y)H_i(y,t)\,dy\right|
\leq
\int_{\varepsilon<|x-y|<3\varepsilon}\frac{1}{|x-y|^{n}}
|H_i(y,t)|\,dy
\\
\lesssim
\dashint_{B(x,3\varepsilon)}
|F(y,t)|\chi_{Q_i}(y)\,dy+
\dashint_{B(x,3\varepsilon)}
\dashint_{Q_i}|F(z,t)|\,dz\chi_{Q_i}(y)\,dy
\lesssim
\dashint_{B(x,3\varepsilon)}
|F(y,t)|\chi_{Q_i}(y)\,dy.
\end{multline*}
It follows that
$$
\left|\int_{|x-y|>\varepsilon}
K(x,y)H(y,t)\,dy\right| \lesssim   \sum_{i\in \N} \frac{\ell(Q_i)^{\delta}}{|x-y_i|^{n+\delta}}\int_{Q_i}
|F(y,t)|\,dy + \dashint_{B(x,3\varepsilon)}
|F(y,t)|\chi_{\Omega_{\lambda}}(y)\,dy.
$$
Taking the supremum over all $\varepsilon>0$, we obtain, 
$$
\mathcal{T}_{*}(H)(x,t) \le \sum_{i\in \N} \frac{\ell(Q_i)^{\delta}}{|x-y_i|^{n+\delta}}\int_{Q_i}
|F(y,t)|\,dy + \mathcal{M}(F)(x,t).
$$
Therefore, by Minkowski inequality
$$
\mathcal{V}_r(\mathcal{T}_{*}(H))(x) \le \sum_{i\in \N}
\frac{\ell(Q_i)^{\delta}}{|x-y_i|^{n+\delta}}\int_{Q_i}
|\mathcal{V}_r(F)(y)|\,dy + \mathcal{V}_r(\mathcal{M}(F))(x).
$$

Consequently, applying Fefferman-Stein  weak type $(1,1)$ inequality  and Chebychev's inequality
\begin{multline*}
\left|\left\{x\in \R^n\setminus\left(\bigcup_{i\in \N}2\sqrt n\, Q_i\right) :\mathcal{V}_r(\mathcal{T}_{*}(H))(x)>\frac{\lambda}{2}\right\}\right|
\\
\lesssim \left|\left\{x\in \R^n:\mathcal{V}_r(\mathcal{M}(F))(x)>\frac{\lambda}{4}\right\}\right|
+
\frac{1}{\lambda}\sum_{i\in \N}
\int_{\R^{n}\setminus 2\sqrt n\, Q_i}
\frac{\ell(Q_i)^{\delta}}{|x-y_i|^{n+\delta}}\,dx\int_{Q_i}
|\mathcal{V}_r(F)(y)|\,dy
\\
\lesssim 
\frac{1}{\lambda}\left(\|\mathcal{V}_r(F)\|_{L^1(\R^n)}
+
\int_{\Omega_\lambda}
|\mathcal{V}_r(F)(y)|\,dy\right)\lesssim
\frac{1}{\lambda}\|\mathcal{V}_rF\|_{L^1(\R^n)}.
\end{multline*}
\qed
\subsection{Proof of Theorem \ref{thm:C-Zoperator}, part $(c)$}
 For $q\le 1$, the space $T^q_{r}$  has an atomic decomposition.  An atom in $T^q_r$ is a measurable function $A(x,t)$ such that there exists a ball $B\subset \R^n$  with $\supp(A)\subset\widehat{B}:=\{(x,t)\in \R^{n+1}_+: d(x,\R^n\setminus B)\geq t\}$, and 
\begin{equation}
\label{eq:qatom}
\left(\iint_{\widehat{B}}|A(x,t)|^r\frac{dx\,dt}{t}\right)^{\frac{1}{r}}\leq |B|^{\frac{1}{r}-\frac{1}{q}}.
\end{equation}
In order to keep things clear we write this as a proposition, whose proof can be found in \cite[Section 8, Proposition 5]{CoifmanMeyerStein}.
\begin{proposition}\label{prop:atoms}
Let $F\in T^q_r$, $0<q\leq 1$ and $1<r<\infty$. Then $F=\sum_{i=1}^{\infty}\lambda_iA_i$, where $A_i$ are $T^q_r$ atoms, $\lambda_i\in \C$, and  $\left(\sum_{i=1}^{\infty}|\lambda_i|^q\right)^{\frac{1}{q}}\lesssim \|F\|_{T^q_r}$. Conversely, any such sum converges in $T^q_{r}$ and 
$ \|\sum_{i=1}^{\infty}\lambda_iA_i\|_{T^q_{r}} \lesssim \left(\sum_{i=1}^{\infty}|\lambda_i|^q\right)^{\frac{1}{q}}.$
\end{proposition}

Let us now introduce, for $0<q\leq 1$ and $1<r<\infty$, a subspace of $T^q_r$ that we denote by $\mathfrak{T}_r^q$. We say that $A$ is a  $\mathfrak{T}^q_r$ atom if it is a $T^q_r$ atom and   satisfies $\int_{\R^n}A(x,t)dx=0$ for a.e. $t>0$. This integral makes sense as \begin{equation*}
\left(\int_{0}^{\infty}\left(\int_{\R^n}|A(x,t)|\,dx\right)^r\frac{dt}{t}\right)^{\frac{1}{r}}
\le
\left(\iint_{\widehat{B}}|A(x,t)|^r\,dx\frac{dt}{t}\right)^{\frac{1}{r}}|B|^{1-\frac{1}{r}}
\le  |B|^{1-\frac{1}{q}}<\infty.
\end{equation*}
We define $\mathfrak{T}^q_r$ as the subspace of 
$F\in T^q_r$ such that $F$ has    an atomic decomposition with $A_{i}$ being $\mathfrak{T}^q_r$ atom and  $\left(\sum_{i=1}^{\infty}|\lambda_i|^q\right)^{\frac{1}{q}}<\infty$. 

The reason to introduce those spaces is because, for $0<q\leq 1$,
we can not obtain boundedness of singular integrals (and in general of Calder\'on-Zygmund operators)
from the tent space $T^q_r$ to itself. If we want to arrive into $T^q_r$, an option is to take functions in $\mathfrak{T}^q_r$. Note that $T^q_r$ atoms, hence $\mathfrak{T}^q_r$ atoms,  belong to $T^r_{r}$.

\begin{lemma}\label{lem:extension} Suppose that $\mathcal{U}: {T}^r_r\rightarrow T^r_r$ is a linear and bounded operator and that there exists $C<\infty$ such that for all 
 $\mathfrak{T}_{r}^{q}$ atom $A$,   $\|\mathcal{U}(A)\|_{T^q_r}\le C$.
 Then, $\mathcal{U}$ has a bounded  extension from $\mathfrak{T}^q_r$ to $ T^q_r$.
 \end{lemma}
 
 \begin{proof} Let   $A$ be a $\mathfrak{T}^q_r$ atom such that $\supp(A)\subset \widehat{B}$, for some ball $B\subset \R^n$.
Defining, for $0<\eta<\rho$, where $\rho$ is the radius of $B$,
$$
A_{\eta}(y,t):=\begin{cases}
A(y,t)\, \textrm{if}\, t> \eta,
\\
0\, \textrm{if}\, t\leq \eta,
\end{cases}\,
$$
we have that $A-A_{\eta}$  are $\mathfrak{T}^q_r$ atoms, uniformly in $\eta$, thus
$$
\|A - A_{\eta}\|_{\mathfrak{T}^q_{r}} \le |B|^{\frac{1}{q}-\frac{1}{r}} \left(\iint_{\widehat{B}}|A(x,t)- A_{\eta}(x,t)|^r\frac{dx\,dt}{t}\right)^{\frac{1}{r}} \to 0
$$
by the dominated convergence theorem. This and the fact that finite linear combinations of $\mathfrak{T}^q_r$ atoms are dense in $\mathfrak{T}^q_r$ by definition,  imply that  the set $E_r$ of compactly supported functions $\varphi$ in $\R^{n+1}_{+}$ that are  $r$ integrable and $ \int_{\R^n}\varphi(x,t)\,dx=0$ {for a.e.} $t>0$ is dense in  $\mathfrak{T}^q_r$. Then, let $F\in E_{r}$ and take a decomposition $F=\sum_{i=0}^{\infty}\lambda_iA_i$, where $A_i$ are $\mathfrak{T}_{r}^{q}$ atoms and  $\left(\sum_{i=1}^{\infty}|\lambda_i|^q\right)^{\frac{1}{q}} \le 2 \|F\|_{\mathfrak{T}^q_{r}}$. Since  the $t$ support of $F$ is contained in some interval $[a,b]$, we may eliminate the atoms associated to balls with radii less than $a$. Following the proof of Theorem 4.9 in \cite{AMcR}, we obtain that the decomposition converges in $T^r_{r}$. Thus we may write
$$
\mathcal{U}(F)= \sum_{i=0}^\infty \lambda_i\,  \mathcal{U}(A_i)
$$
and use the hypothesis to conclude that $\|\mathcal{U}(F)\|_{T^q_{r}}\le 2 C\|F\|_{\mathfrak{T}^q_{r}}$. By density, we conclude the argument. 
\end{proof}

We say that function $M$ is a $T^q_r$ molecule if there exists a ball $B\subset \R^n$ such that, for some $\varepsilon>0$,
\begin{align*}
\left(\iint_{\widehat{4B}}|M(x,t)|^r\frac{dx\,dt}{t}\right)^{\frac{1}{r}}\leq |4B|^{\frac{1}{r}-\frac{1}{q}}
\end{align*}
and, for all $j\geq 2$,
\begin{align*}
\left(\iint_{\widehat{C_{j}}}|M(x,t)|^r\frac{dx\,dt}{t}\right)^{\frac{1}{r}}\leq 2^{-({j+1})\varepsilon}|2^{j+1}B|^{\frac{1}{r}-\frac{1}{q}},
\end{align*}
where we define $\widehat{C_{j}}:=\widehat{2^{j+1}B}\setminus
\widehat{2^{j}B}$ and $\widehat{C_{1}}=\widehat{4B}$ . 
 By writing $M= \sum_{j\ge 1} {\bf 1}_{ \widehat{C_{j}}} M$ and observing that the functions ${\bf 1}_{ \widehat{C_{j}}} M$ are $T^q_{r}$ atoms up to factor $2^{-(j+1)\varepsilon}$, we obtain $\|M\|_{T^{q}_{r}} \le \left(\sum_{j\geq 1}|2^{-(j+1)\varepsilon }|^q\right)^{\frac{1}{q}}$.

Let us finally prove part $(c)$ of Theorem \ref{thm:C-Zoperator}. We  follow the same scheme as in \cite{CoifmanWeiss} and show that Calder\'on-Zygmund operators of order $\delta \in (0,1]$ apply $\mathfrak{T}^q_r$ atoms  to $T^q_r$ molecules, provided that $q>\frac{n}{n+\delta}$, up to a constant that depends uniquely on $\delta, n, r, q$ and the properties of the operator. From the previous lemma, this suffices.

Let $A$ is a $\mathfrak{T}^q_r$ atom. Let $B$ be a ball such that
$\supp A\subset \widehat{B}$, and
\begin{align*}
\left(\iint_{\widehat{B}}|A(x,t)|^r\right)^{\frac{1}{r}}\leq |B|^{\frac{1}{r}-\frac{1}{q}}.
\end{align*}
We shall show that, for $\varepsilon=n+\delta-\frac{n}{q}$ which is positive since $q>\frac{n}{n+\delta}$,
\begin{list}{$(\theenumi)$}{\usecounter{enumi}\leftmargin=1cm \labelwidth=1cm\itemsep=0.2cm\topsep=.0cm \renewcommand{\theenumi}{\arabic{enumi}}}
\item $\left(\iint_{\widehat{4B}}|\mathcal{T}(A(\cdot,t))(x)|^r\right)^{\frac{1}{r}}\lesssim |4B|^{\frac{1}{r}-\frac{1}{q}}$;

\item for $j\geq 2$, $\left(\iint_{\widehat{C_j}}|\mathcal{T}(A(\cdot,t))(x)|^r\right)^{\frac{1}{r}}\lesssim 2^{-({j+1})\varepsilon}|B_{j+1}|^{\frac{1}{r}-\frac{1}{q}}$.

\end{list}
For each $j\geq 2$, denote by $r_j:=2^{j}r_B$ and $B_j:=B(x_B,r_j)$. Besides, recall that $\widehat{C_1}:=\widehat{B_{2}}$ and $\widehat{C_j}:=\widehat{B_{j+1}}\setminus \widehat{B_j}$, for all $j\geq 2$.

We start by proving $(1)$. Since $\mathcal{T}$ is bounded in $T_{r}^r$, we have that
\begin{align*}
\left(\iint_{\widehat{4B}}|\mathcal{T}(A(\cdot,t))(x)|^r\frac{dx\,dt}{t}\right)^{\frac{1}{r}}
\lesssim
\left(\iint_{\widehat{B}}|A(x,t)|^r\frac{dx\,dt}{t}\right)^{\frac{1}{r}}
\lesssim
|4B|^{\frac{1}{r}-\frac{1}{q}}.
\end{align*}
On the other hand, for $j\geq 2$, because $A(x,t)=0$ for $t>r_{B}$, the radius of $B$,  
\begin{equation*}
\left(\iint_{\widehat{C_j}}|\mathcal{T}(A(\cdot,t))(x)|^{r}\frac{dx\,dt}{t}\right)^{\frac{1}{r}}
\le
\left(\int_{0}^{r_{B}}\int_{B_{j+1}\setminus B_{j-1}}|\mathcal{T}(A(\cdot,t))(x)|^r\frac{dx\,dt}{t}\right)^{\frac{1}{r}}.
\end{equation*}
Now, applying the fact that $\int_{\R^n}A(x,t)\,dx=0$ for a.e. $t>0$, and the property \eqref{kernel2} of the kernel $K$, we obtain that
\begin{multline*}
I
\leq
\left(\int_{0}^{r_{B}}\int_{r_{j-1}\le |x-x_B|< r_{j+1}}\left|\int_{\R^n}K(x,y)A(y,t)\,dy\right|^r\frac{dx\,dt}{t}\right)^{\frac{1}{r}}
\\
=
\left(\int_{0}^{r_{B}}\int_{r_{j-1}\le |x-x_B|< r_{j+1}}\left|\int_{\R^n}(K(x,y)-K(x,x_B))A(y,t)\,dy\right|^r\frac{dx\,dt}{t}\right)^{\frac{1}{r}}
\\
\lesssim
\left(\int_{0}^{r_{B}}\int_{r_{j-1}\le |x-x_B|< r_{j+1}}\left(\int_{\R^n}\frac{|x_B-y|^{\delta}}{|x-x_B|^{n+\delta}}|A(y,t)|\,dy\right)^r\frac{dx\,dt}{t}\right)^{\frac{1}{r}}
\\
\lesssim
\left(\int_{0}^{r_{B}}\int_{r_{j-1}\le |x-x_B|< r_{j+1}}\dashint_{B}|A(y,t)|^r\,dy\frac{dx\,dt}{t}\right)^{\frac{1}{r}} 2^{-(j+1)(n+\delta)}
\\
\lesssim
\left(\int_{0}^{r_{B}}\int_{B}|A(y,t)|^r\frac{dy\,dt}{t}\right)^{\frac{1}{r}} 2^{-(j+1)\left(n\left(1-\frac{1}{r}\right)+\delta\right)}
\\
\lesssim
 2^{-(j+1)\left(n\left(1-\frac{1}{r}\right)+\delta\right)}|B|^{\frac{1}{r}-\frac{1}{q}}
= 2^{-(j+1)\left(n+\delta-\frac{n}{q}\right)}|2^{j+1}B|^{\frac{1}{r}-\frac{1}{q}}.
\end{multline*}
This shows $(2)$. 
\qed
\subsection{Proof of Theorem \ref{thm:C-Zoperator}, part $(d)$}
Remark that if $M$ is a $T^q_r$ molecule, then
\begin{multline*}
\left(\int_{0}^{\infty}\left(\int_{\R^n}|M(x,t))|\,dx\right)^r\frac{dt}{t}\right)^{\frac{1}{r}}
\lesssim
\sum_{j\geq 1}
\left(\int_{0}^{\infty}\left(\int_{\R^n}\chi_{\widehat{C_j}}(x,t)|M(x,t)|\,dx\right)^r\frac{dt}{t}\right)^{\frac{1}{r}}
\\
\lesssim
\sum_{j\geq 1}
\left(\iint_{\widehat{C_j}}|M(x,t)|^r\,dx\frac{dt}{t}\right)^{\frac{1}{r}}|B_{j+1}|^{1-\frac{1}{r}}
\lesssim
\sum_{j\geq 1} 2^{-(j+1)\varepsilon}
|B_{j+1}|^{1-\frac{1}{q}}
\lesssim |B|^{1-\frac{1}{q}}<\infty
\end{multline*}
as $1-\frac{1}{q}\le 0$.
Therefore, if, in addition,  $\int_{\R^n}M(x,t)\,dx=0$, for a.e. $t>0$, we say that $M$ is a $\mathfrak{T}^q_{r}$ molecule.  A $\mathfrak{T}^q_r$ molecule can be written as a series of $\mathfrak{T}^q_r$ atoms. We see that in the next result.
\begin{proposition}\label{prop:moleculesumatoms} There exists a constant $C<\infty$ such that 
given  a $\mathfrak{T}^q_r$ molecule $M$,
we have that $M\in \mathfrak{T}^q_r$, with
$\|M\|_{\mathfrak{T}^q_r}\leq C.$
\end{proposition}
\begin{proof}
Let $M$ be  a $\mathfrak{T}^q_r$ molecule with associated ball $B=B(x_B,r_B)$.
Following the notation in the previous proof, write
$$
M(x,t)=\sum_{j=1}^{\infty}\left(M(x,t)\chi_{\widehat{C_j}}(x,t)-
\int_{\R^n}\chi_{\widehat{C_j}}(y,t)M(y,t)\,dy\frac{\chi_{{B_{j+1}}}(x)}{|B_{j+1}|}\right)
+
\sum_{j=1}^{\infty}
\int_{\R^n}\chi_{\widehat{C_j}}(y,t)M(y,t)\,dy\frac{\chi_{{B_{j+1}}}(x)}{|B_{j+1}|}.
$$
For all $j\geq 1$, define
$$
\alpha_j(x,t):=
M(x,t)\chi_{\widehat{C_j}}(x,t)-
\int_{\R^n}\chi_{\widehat{C_j}}(y,t)M(y,t)\,dy\frac{\chi_{{B_{j+1}}}(x)}{|B_{j+1}|},
$$
and observe that $\supp \alpha_j\subset B_{j+1}\times (0,r_{j+1}] \subset \widehat{B_{j+2}}$  and 
$$
\int_{\R^n}\alpha_j(x,t)\, dx=\int_{\R^n}\chi_{\widehat{C_j}}(y,t)M(y,t)\,dy\left(1-
\int_{\R^n}\frac{\chi_{B_{j+1}}(x)}{|B_{j+1}|}\,dx\right)=0.
$$
Besides,
\begin{multline*}
\left(\iint_{\widehat{B_{j+2}}}|\alpha_j(x,t)|^r\frac{dx\,dt}{t}\right)^{\frac{1}{r}}
\leq
\left(\iint_{\widehat{C_{j}}}|M(x,t)|^r\frac{dx\,dt}{t}\right)^{\frac{1}{r}}
\\
+
\left(\iint_{\widehat{B_{j+2}}}\chi_{B_{j+1}}(x)\left(\frac{1}{|B_{j+1}|}\int_{B_{j+1}}\chi_{\widehat{C_j}}(y,t)|M(y,t)|\,dy\right)^r\frac{dx\,dt}{t}\right)^{\frac{1}{r}}
\leq 2 \left(\iint_{\widehat{C_{j}}}|M(x,t)|^r\frac{dx\,dt}{t}\right)^{\frac{1}{r}}
\\
\leq 
2^{-(j+1)\varepsilon+1}|B_{j+1}|^{\frac{1}{r}-\frac{1}{q}}
=c2^{-{j}\varepsilon}|B_{j+2}|^{\frac{1}{r}-\frac{1}{q}},
\end{multline*}
where $c$  depends on $\varepsilon,r,q$ only.
Therefore, 
$A_j:=\frac{2^{j\varepsilon}}{c}\alpha_j$ is a $\mathfrak{T}^q_r$ atom, for all $j\geq 1$.

On the other hand,  note that 
\begin{multline*}
\sum_{j=1}^{\infty}
\int_{\R^n}\chi_{\widehat{C_j}}(y,t)M(y,t)\,dy\frac{\chi_{{B_{j+1}}}(x)}{|B_{j+1}|}
=
\int_{\R^n}\chi_{\widehat{B_{2}}}(y,t)M(y,t)\,dy
\frac{\chi_{B_2}(x)}{|B_2|}
\\
+
\sum_{j=2}^{\infty}\left(
\int_{\R^n}\chi_{\widehat{B_{j+1}}}(y,t)M(y,t)\,dy-
\int_{\R^n}\chi_{\widehat{B_{j}}}(y,t)M(y,t)\,dy\right)
\frac{\chi_{{B_{j+1}}}(x)}{|B_{j+1}|}
\\
=
\sum_{j=1}^{\infty}
\int_{\R^n}\chi_{ \widehat{B_{j+1}}}(y,t)
M(y,t)\,dy\left(\frac{\chi_{{B_{j+1}}}(x)}{|B_{j+1}|}-\frac{\chi_{{B_{j+2}}}(x)}{|B_{j+2}|}\right).
\end{multline*}
Then, considering
$$
\beta_j(x,t):=
\int_{\R^n}\chi_{\widehat{B_{j+1}}}(y,t)
M(y,t)\,dy\left(\frac{\chi_{{B_{j+1}}}(x)}{|B_{j+1}|}-\frac{\chi_{{B_{j+2}}}(x)}{|B_{j+2}|}\right),
$$
we have that $\supp \beta_j\subset \widehat{B_{j+3}}$, and that 
$$
\int_{\R^n}\beta_j(x,t)dx=0.
$$
Besides, since, for a. e. $t>0$,
$$
\int_{\R^n}M(y,t)\,dy=0, 
$$
then, for each $j\geq 1$,
$$
\int_{\R^n}\chi_{\R^{n+1}_+\setminus \widehat{B_{j+1}}}(y,t)M(y,t)\,dy=-
\int_{\R^n}\chi_{\widehat{B_{j+1}}}(y,t)M(y,t)\,dy, \quad \textrm{for a.e. }t>0.
$$
This, together with the fact that
$$
\int_{\R^n}\chi_{\widehat{B_{j+1}}}(y,t)M(y,t)\,dy=
\chi_{(0,r_{j+1})}(t)\int_{\R^n}\chi_{\widehat{B_{j+1}}}(y,t)M(y,t)\,dy,
$$
 gives, for a.e. $t>0$,
$$
\int_{\R^n}\chi_{\widehat{B_{j+1}}}(y,t)M(y,t)\,dy=-
\chi_{(0,r_{j+1})}(t)\int_{\R^n}\chi_{\R^{n+1}_+\setminus \widehat{B_{j+1}}}(y,t)M(y,t)\,dy.
$$ 
Hence, for all $j\geq 1$,
$$
\beta_j(x,t)=
\chi_{(0,r_{j+1})}(t)\int_{\R^n}\chi_{\R^{n+1}_+\setminus \widehat{B_{j+1}}}(y,t)M(y,t)\,dy\left(\frac{\chi_{{B_{j+2}}}(x)}{|B_{j+2}|}-\frac{\chi_{{B_{j+1}}}(x)}{|B_{j+1}|}\right),\quad \textrm{for a.e. }t>0.
$$
Therefore,
\begin{multline*}
\left(\iint_{\widehat{B_{j+3}}}|\beta_j(x,t)|^r\frac{dx\,dt}{t}\right)^{\frac{1}{r}}
\lesssim 
\sum_{i\geq j+1}\frac{|B_{i+1}|}{|B_{j+1}|}
\left(\iint_{\widehat{B_{j+3}}}\chi_{B_{j+2}}(x)\left|\frac{1}{|B_{i+1}|}\int_{\R^n}\chi_{\widehat{C_i}}(y,t)
M(y,t)\,dy\right|^r\frac{dx\,dt}{t}\right)^{\frac{1}{r}}
\\
\lesssim \sum_{i\geq j+1}
 \left(\frac{|B_{i+1}|}{|B_{j+1}|}\right)^{1-\frac{1}{r}}
\left(\iint_{\widehat{C_{i}}}
|M(y,t)|^r\frac{dy\,dt}{t}\right)^{\frac{1}{r}}
\lesssim 
\sum_{i\geq j+1}2^{-(i+1)\varepsilon}
 |B_{j+1}|^{\frac{1}{r}-1}
 |B_{i+1}|^{1-\frac{1}{q}} 
 \\
 \lesssim
  |B_{j+3}|^{\frac{1}{r}-\frac{1}{q}}\sum_{i\geq j+1}2^{-(i+1)\varepsilon}
 \leq c' 2^{-j\varepsilon}
 |B_{j+3}|^{\frac{1}{r}-\frac{1}{q}},
\end{multline*}
where $c'$ depends on $\varepsilon,r,q$ only.
Hence, 
$A'_j(x,t):=\frac{2^{j\varepsilon}}{c'}\beta_j$ is
a $\mathfrak{T}^q_r$ atom. 

Therefore, we have shown that $M= \sum_{j\ge 1} c2^{-j\varepsilon} A_{j}+ \sum_{j\ge 1} c'2^{-j\varepsilon} A'_{j}$, which evidently shows that $M\in \mathfrak{T}^q_r$ with norm bounded by  $(c+c')\left(\sum_{j\geq 1}|2^{-j\varepsilon q}|^q\right)^{\frac{1}{q}}.$
\end{proof}

Let us finally show that if $\mathcal{T}$ is a Calder\'on-Zygmund operator, then  $\mathcal{T}$ applies $\mathfrak{T}^q_r$ atoms to $\mathfrak{T}^q_r$ molecules, up to a uniform constant. Note that, from the above proposition,  and an adaptation of Lemma \ref{lem:extension},  this is enough to conclude the proof.

From the part (c) of the proof, we already know that $\mathcal{T}$ applies $\mathfrak{T}^q_r$ atoms to ${T}^q_r$ molecules, up to a uniform constant.  It remains to show $\int_{\R^n}\mathcal{T}(A(\cdot,t))(x)\,dx=0$.  Note for almost every $t>0$, $A(\cdot, t)$ is a multiple of an atom in the Hardy space $H^1(\R^n)$. Indeed, its support is contained in $B$, it is in $L^r(B)$ with $r>1$ and has mean value 0. We knew that  $\mathcal{T}(A(\cdot,t))\in L^1(\R^n)$ since $\mathcal{T}(A)$ has been shown to be a $T^q_{r}$ molecule. Thus, $\int_{\R^n}\mathcal{T}(A(\cdot,t))(x)\,dx=0$ as  $\mathcal{T}^*(1)=0$. 
\section{Riesz potentials and fractional maximal functions}
For $0<\alpha <n$, consider the Riesz potential 
\begin{align*}
\mathcal{I}_{\alpha}(f)(x):=\frac{1}{\gamma(\alpha)}\int_{\R^n}\frac{1}{|x-z|^{n-\alpha}}f(z)\,dz,
\end{align*}
where $\gamma(\alpha)=\pi^{\frac{n}{2}}2^{\alpha}\Gamma(\alpha/2)/\Gamma\left(\frac{n-\alpha}{2}\right)$,  and the fractional maximal function 
\begin{align*}
\mathcal{M}_{\alpha}(f)(x)=\sup_{\tau>0}\tau^{\frac{\alpha}{n}}\dashint_{B(x,\tau)}|f(y)|dy.
\end{align*}
Note that 
\begin{align}\label{comparisonmaxriesz}
\mathcal{M}_{\alpha}(f)(x)\leq V_n^{-1}\mathcal{I}_{\alpha}(|f|)(x), \textrm{ for all }x\in \R^{n},
\end{align}
where $V_n$ is the volume of the unit ball in $\R^n$.

Consequently, it is enough to prove Theorem \ref{thm:rieszpotential} for Riesz potentials. Let us start by proving the following pointwise inequality.
\begin{lemma}\label{lemma:rieszpotential}
Let $0<\alpha<n$, $1< \vartheta<r<\infty$, and $\frac{\alpha}{n}=\frac{1}{\vartheta}-\frac{1}{r}$.  Then, for any $x\in \R^n$, $t>0$, if $f$ is locally  $\vartheta$ integrable, 
\begin{align*}
\left(\dashint_{B(x,t)}|\mathcal{I}_{\alpha}(f)(y)|^{r}dy\right)^{\frac{1}{r}}
\lesssim
t^{n\left(\frac{1}{\vartheta}-\frac{1}{r}\right)}\left(\dashint_{B(x,5t)}|f(y)|^{\vartheta}dy\right)^{\frac{1}{\vartheta}}
+
\mathcal{I}_{\alpha}\left(\dashint_{B(\cdot,t)}|f(z)|dz\right)(x).
\end{align*}
\end{lemma} 
\begin{proof}
For each $x\in \R^n$ and $t>0$, split the support of $f$ into $B(x,5t)$ and $\R^n\setminus B(x,5t)$. Then,
\begin{multline*}
\left(\dashint_{B(x,t)}|\mathcal{I}_{\alpha}(f)(y)|^{r}dy\right)^{\frac{1}{r}}
\leq
\left(\dashint_{B(x,t)}|(\mathcal{I}_{\alpha}(\chi_{B(x,5t)}f))(y)|^{r}dy\right)^{\frac{1}{r}}
\\
+
\left(\dashint_{B(x,t)}\left|\int_{|x-z|>5t}\frac{1}{|y-z|^{n-\alpha}}f(z)\,dz\right|^{r}dy\right)^{\frac{1}{r}}
=:I+II.
\end{multline*}
On the one hand, using that $\mathcal{I}_{\alpha}:L^{\vartheta}(\R^n)\rightarrow L^{r}(\R^n)$ (see \cite[Theorem 1, p.119]{Stein}), obtain that
\begin{align*}
I
\lesssim 
t^{n\left(\frac{1}{\vartheta}-\frac{1}{r}\right)}\left(\dashint_{B(x,5t)}|f(y)|^{\vartheta}dy\right)^{\frac{1}{\vartheta}}.
\end{align*} 
On the other hand, 
\begin{align*}
II
&
\lesssim 
\left(\dashint_{B(x,t)}\left(\int_{|x-z|>5t}\frac{1}{|y-z|^{n-\alpha}}|f(z)|\dashint_{B(z,t)}\,d\xi\,dz\right)^{r}dy\right)^{\frac{1}{r}}
\\
&
\lesssim 
\left(\dashint_{B(x,t)}\left(\int_{|x-\xi|>4t}\dashint_{B(\xi,t)}\frac{1}{|y-z|^{n-\alpha}}|f(z)|\,dz\,d\xi\right)^{r}dy\right)^{\frac{1}{r}}
\\
&
\lesssim 
\left(\dashint_{B(x,t)}\left(\int_{|x-\xi|>4t}\frac{1}{|x-\xi|^{n-\alpha}}\dashint_{B(\xi,t)}\left(\frac{|x-\xi|}{|y-z|}\right)^{n-\alpha}|f(z)|\,dz\,d\xi\right)^{r}dy\right)^{\frac{1}{r}}
\\
&
\lesssim 
\left(\dashint_{B(x,t)}\left(\int_{|x-\xi|>4t}\frac{1}{|x-\xi|^{n-\alpha}}\dashint_{B(\xi,t)}|f(z)|\,dz\,d\xi\right)^{r}dy\right)^{\frac{1}{r}}
\\
&
=
\int_{|x-\xi|> 4t}\frac{1}{|x-\xi|^{n-\alpha}}\dashint_{B(\xi,t)}|f(z)|\,dz\,d\xi
\leq
\mathcal{I}_{\alpha}\left(\dashint_{B(\cdot,t)}|f(z)|dz\right)(x).
\end{align*} 
\end{proof}
\subsection{Proof of Theorem \ref{thm:rieszpotential}}

Let $F\in T^p_{r}$. Taking $\vartheta=\frac{nr}{\alpha r+n}$ in Lemma \ref{lemma:rieszpotential},
we obtain that
\begin{multline*}
\|\mathcal{I}_{\alpha}(F)\|_{T^q_{r}}
\lesssim 
\left(\int_{\R^n}\left(\int_0^{\infty}t^{n\left(\frac{r}{\vartheta}-1\right)}\left(\dashint_{B(x,5t)}
|F(y,t)|^{\vartheta} dy\right)^{\frac{r}{\vartheta}}\frac{dt}{t}\right)^{\frac{q}{r}}dx\right)^{\frac{1}{q}}
\\
+
\left(\int_{\R^n}\left(\int_0^{\infty}\left(\mathcal{I}_{\alpha}
\left(\dashint_{B(\cdot,t)}|F(y,t)|dy\right)(x)\right)^{r}\frac{dt}{t}\right)^{\frac{q}{r}}dx\right)^{\frac{1}{q}}=:I+II.
\end{multline*}
Since $r>\vartheta$, applying successively Jensen's inequality,  \cite[Theorem 2.19]{Amenta} for $s_1=\frac{1}{r}-\frac{1}{\vartheta}$, $s_0=0$, $p_0=p$, $p_{1}=q$, and $q=r$, and \cite[Section 3, Proposition 4]{CoifmanMeyerStein} (we use this proposition for $r$ instead of $2$, but the proof is the same), 
\begin{multline*}
I
\lesssim
\left(\int_{\R^n}\left(\int_0^{\infty}\int_{B(x,5t)}
\left|t^{n\left(\frac{1}{\vartheta}-\frac{1}{r}\right)}F(y,t)\right|^{r}\frac{dy\,dt}{t^{n+1}}\right)^{\frac{q}{r}}dx\right)^{\frac{1}{q}}
\\
\lesssim
\left(\int_{\R^n}\left(\int_0^{\infty}\int_{B(x,5t)}
|F(y,t)|^{r}\frac{dy\,dt}{t^{n+1}}\right)^{\frac{p}{r}}dx\right)^{\frac{1}{p}}
\lesssim
\|F\|_{T^p_{r}}.
\end{multline*}
Finally, to estimate $II$, we shall proceed by extrapolation. We first recall some definitions. We say that a weight $w$ is a $A_{\tau,s}$ weight, for $1<\tau\leq s<\infty$, if it satisfies for every $B\subset \R^n$ that
$$
\left(\dashint_{B}w(x)^s\,dx\right)^{\frac{1}{s}}
\left(\dashint_{B}w(x)^{-\tau'}\,dx\right)^{\frac{1}{\tau'}}\leq C.
$$ 
Now, since $0<\alpha<n$ and $1<\vartheta<\frac{n}{\alpha}$ with  $\frac{1}{\vartheta}-\frac{1}{r}=\frac{\alpha}{n}$, by \cite[Theorem 4]{MuckenhouptWheede} for all $w\in A_{\vartheta,r}$ 
we have that $\mathcal{I}_{\alpha}:L^{\vartheta}(w^\vartheta)\rightarrow L^{r}(w^{r})$. This and Minkowski's integral inequality imply
\begin{multline*}
\left(\int_{\R^n}\int_0^{\infty}
\left|\mathcal{I}_{\alpha}\left(\dashint_{B(\cdot,t)}|F(y,t)|dy\right)(x)\right|^{r}\frac{dt}{t}w(x)^{r}dx\right)^{\frac{1}{r}}
\\
\lesssim
\left(\int_0^{\infty}\int_{\R^n}\left|\mathcal{I}_{\alpha}
\left(\dashint_{B(\cdot,t)}|F(y,t)|dy\right)(x)\right|^{r}w(x)^{r}dx\frac{dt}{t}\right)^{\frac{1}{r}}
\\
\lesssim
\left(\int_0^{\infty}\left(\int_{\R^n}\left(\dashint_{B(x,t)}|F(y,t)|dy\right)^{\vartheta}w(x)^{\vartheta}dx\right)^{\frac{r}{\vartheta}}\frac{dt}{t}\right)^{\frac{1}{r}}
\\
\lesssim
\left(\int_{\R^n}\left(\int_0^{\infty}\left(\dashint_{B(x,t)}|F(y,t)|dy\right)^{r}\frac{dt}{t}\right)^{\frac{\vartheta}{r}}w(x)^{\vartheta}dx
\right)^{\frac{1}{\vartheta}}
\\
\lesssim
\left(\int_{\R^n}\left(\int_0^{\infty}\int_{B(x,t)}|F(y,t)|^{r}\frac{dy\,dt}{t^{n+1}}\right)^{\frac{\vartheta}{r}}w(x)^{\vartheta}dx
\right)^{\frac{1}{\vartheta}}.
\end{multline*}
Then, since $1<\vartheta<r<\infty$ and $1<p<q<\infty$ with $\frac{1}{p}-\frac{1}{q}=\frac{1}{\vartheta}-\frac{1}{r}$, applying  \cite[Theorem 3.23]{CruzMartellPerez}, we have that, for all $w_0\in A_{p,q}$, and $F\in T^p_{r}$, 
\begin{multline*}
\left(\int_{\R^n}\left(\int_0^{\infty}\left(
\mathcal{I}_{\alpha}\left(\dashint_{B(\cdot,t)}|F(y,t)|dy\right)(x)\right)^{r}\frac{dt}{t}\right)^{\frac{q}{r}}w_0(x)^{q}dx\right)^{\frac{1}{q}}
\\
\lesssim
\left(\int_{\R^n}\left(\int_0^{\infty}\int_{B(x,t)}|F(y,t)|^{r}\frac{dy\,dt}{t^{n+1}}\right)^{\frac{p}{r}}w_0(x)^{p}dx
\right)^{\frac{1}{p}}.
\end{multline*}
In particular for $w_0\equiv 1$, we have that $w_0\in A_{p,q}$. Hence,
\begin{align*}
II
\lesssim
\left(\int_{\R^n}\left(\int_0^{\infty}\int_{B(x,t)}|F(y,t)|^{r}\frac{dy\,dt}{t^{n+1}}\right)^{\frac{p}{r}}dx
\right)^{\frac{1}{p}}=\|F\|_{T^p_{r}}.
\end{align*}
\qed
\section{Riesz transform}\label{sec:RT}
Consider a second order divergence form elliptic operator $L$ which is defined as
\begin{align*}
L f
=
-\div(A\,\nabla f)
\end{align*}
and
is understood in the standard weak sense as a maximal-accretive operator on $L^2(\R^n,dx)$ with domain $\mathcal{D}(L)$ by means of a
sesquilinear form, and 
where $A$ is an $n\times n$ matrix of complex and
$L^\infty$-valued coefficients defined on $\R^n$. We assume that
this matrix satisfies the following ellipticity (or \lq\lq
accretivity\rq\rq) condition: there exist
$0<\lambda\le\Lambda<\infty$ such that
$$
\lambda\,|\xi|^2
\le
\Re A(x)\,\xi\cdot\bar{\xi}
\quad\qquad\mbox{and}\qquad\quad
|A(x)\,\xi\cdot \bar{\zeta}|
\le
\Lambda\,|\xi|\,|\zeta|,
$$
for all $\xi,\zeta\in\mathbb{C}^n$ and almost every $x\in \R^n$. We have used the notation
$\xi\cdot\zeta=\xi_1\,\zeta_1+\cdots+\xi_n\,\zeta_n$ and therefore
$\xi\cdot\bar{\zeta}$ is the usual inner product in $\mathbb{C}^n$. Note
that then
$A(x)\,\xi\cdot\bar{\zeta}=\sum_{j,k}a_{j,k}(x)\,\xi_k\,\bar{\zeta_j}$.

We recall some facts regarding the operator $-L$. This operator generates a $C^0$-semigroup $\{e^{-t L}\}_{t>0}$ of contractions on $L^2(\R^n)$ which is called the heat semigroup. As in \cite{Auscher} and \cite{AuscherMartell:II}, we denote by $(p_-(L),p_+(L))$ the maximal open interval on which this semigroup $\{e^{-tL}\}_{t>0}$ is uniformly bounded on $L^p(\R^n)$:
\begin{align}\label{p-}
p_-(L) &:= \inf\left\{p\in(1,\infty): \sup_{t>0} \|e^{-tL}\|_{L^p(\R^n)\rightarrow L^p(\R^n)}< \infty\right\},
\\[4pt]
p_+(L)& := \sup\left\{p\in(1,\infty) : \sup_{t>0} \|e^{-tL}\|_{L^p(\R^n)\rightarrow L^p(\R^n)}< \infty\right\}.
\label{p+}
\end{align}
Moreover, we denote by $(q_-(L),q_+(L))$ the maximal open interval on which the gradient of the heat semigroup, i.e.  $\{t\nabla_y e^{-t^2L}\}_{t>0}$, is uniformly bounded on $L^p(\R^n)$:
\begin{align}\label{q-q+}
q_-(L) &:= \inf\left\{p\in(1,\infty): \sup_{t>0} \|t\nabla_y e^{-t^2L} \|_{L^p(\R^n)\rightarrow L^p(\R^n)}< \infty\right\},
\\[4pt]
q_+(L)& := \sup\left\{p\in(1,\infty) : \sup_{t>0} \|t\nabla_y e^{-t^2L} \|_{L^p(\R^n)\rightarrow L^p(\R^n)}< \infty\right\}.
\end{align}
From \cite{Auscher} (see also \cite{AuscherMartell:II}) we know that $p_-(L)=1$ and  $p_+(L)=\infty$ if $n=1,2$; and if $n\ge 3$ then $p_-(L)<\frac{2\,n}{n+2}$ and $p_+(L)>\frac{2\,n}{n-2}$. Moreover, $q_-(L)=p_-(L)$, $ q_+(L)\le p_+(L)$, and we always have $q_+(L)>2$, with $q_+(L)=\infty$ if $n=1$.

We shall obtain a pointwise inequality for the Riesz transform taking a generalized version of two inequalities that appear in \cite[Lemma 4.8 and (4.6)]{Auscher}. These are:
\begin{lemma}\label{lemma1:riesz}
For every ball $B$, with radius $r_B$, and $q_-(L)<r<q_{+}(L)$,
\begin{align*}
\|\nabla L^{-\frac{1}{2}}(I-e^{-r_B^2L})^Mh\|_{L^r(B)}
\leq 
|B|^{\frac{1}{r}}\sum_{j\geq 1}g(j)\left(\dashint_{2^{j+1}B}|h|^r\right)^{\frac{1}{r}},
\end{align*}
with $g(j)=C2^{j\frac{n}{2}}4^{-jM}$, where $M\in \N$ is arbitrary and $C$ depends on $M$. 
\end{lemma}
\begin{lemma}\label{lemma2:riesz}
For every ball $B$, with radius $r_B$, any constant $k>0$, and $q_-(L)<p_0\leq r<q_+(L)$,
\begin{align*}
\left(\dashint_{B}|\nabla e^{-kr_B^2L} h|^{r}\right)^{\frac{1}{r}}
\leq 
\sum_{j\geq 1}g(j)\left(\dashint_{2^{j+1}B}|\nabla h|^{p_0}\right)^{\frac{1}{p_0}},
\end{align*}
with $\sum_{j\geq 1}g(j)<\infty$.
\end{lemma}
It is in the first inequality that was used the integral representation $\nabla L^{-\frac{1}{2}}h = \pi^{-1/2}\int_{0}^\infty \nabla e^{-tL}h \frac{dt}{\sqrt t}$ for appropriate $h$ (to replace for  the kernel representation in the case of Calder\'on-Zygmund operators).

From these two results we have the following corollary.

\begin{corollary}\label{cor:riesz} Let $q_-(L)<p_0< r<q_+(L)$.
For every $x\in \R^n$ and $t>0$ and $f\in L^r(\R^n)$. 
\begin{multline*}
\left(\dashint_{B(x,t)}|\nabla L^{-\frac{1}{2}}(f)(y)|^r\,dy\right)^{\frac{1}{r}}
\\
\lesssim 
\sum_{j\geq 1}4^{-jM}\left(\int_{B(x,2^{j+1}t)}|f(y)|^r\,\frac{dy}{t^n}\right)^{\frac{1}{r}}
+
\sum_{k=1}^{M}C_{k,M}
\mathcal{M}_{p_0}\left(\nabla L^{-\frac{1}{2}}e^{-\frac{kt^2}{2}L}(f)\right)(x),
\end{multline*}
where $M\in \N$ is arbitrarily large and $\mathcal{M}_{p_0}(f):=(\mathcal{M}(|f|^{p_0}))^{\frac{1}{p_0}}$.
\end{corollary}
\begin{proof}
Fix $x\in \R^n$, $t>0$ and  $M\in \N$ arbitrarily large. We have that
\begin{multline*}
\left(\dashint_{B(x,t)}|\nabla L^{-\frac{1}{2}}(f)(y)|^r\,dy\right)^{\frac{1}{r}}
\\
\lesssim 
\left(\dashint_{B(x,t)}|\nabla L^{-\frac{1}{2}}(I-e^{-t^2L})^M(f)(y)|^r\,dy\right)^{\frac{1}{r}}
+
\left(\dashint_{B(x,t)}|\nabla L^{-\frac{1}{2}}A_{t,M}(f)(y)|^r\,dy\right)^{\frac{1}{r}}
=:I+II,
\end{multline*}
 where $A_{t,M}:=I-(I-e^{-t^2L})^M$. Then, applying Lemma \ref{lemma1:riesz} for $B=B(x,t)$ and $h=f$, we obtain that
\begin{align*}
I
\lesssim
\sum_{j\geq 1}4^{-jM}
\left(\int_{B(x,2^{j+1}t)}|f(y)|^r\frac{dy}{t^n}\right)^{\frac{1}{r}}.
\end{align*}
As for the estimate of $II$, note that expanding the binomial expression, we have that $A_{t,M}=\sum_{k=1}^{M}C_{k,M}e^{-kt^2L}$. Then,
 applying Lemma \ref{lemma2:riesz} for $B=B(x,t)$ and $h=L^{-\frac{1}{2}}e^{-\frac{kt^2}{2}L}f$, 
\begin{multline*}
II\lesssim \sum_{k=1}^{M}C_{k,M}\left(\dashint_{B(x,t)}|\nabla e^{-\frac{kt^2L}{2}} L^{-\frac{1}{2}}e^{-\frac{kt^2L}{2}}(f)(y)|^r\,dy\right)^{\frac{1}{r}}
\\
\lesssim
\sum_{k=1}^{M}C_{k,M}
\sum_{j\geq 1}g(j)
\left(\dashint_{B(x,2^{j+1}t)}|\nabla L^{-\frac{1}{2}}e^{-\frac{kt^2}{2}L}(f)(y)|^{p_0}dy\right)^{\frac{1}{p_0}}
\\
\lesssim
\sum_{k=1}^{M}C_{k,M}
\mathcal{M}_{p_0}\left(\nabla L^{-\frac{1}{2}}e^{-\frac{kt^2}{2}L}(f)\right)(x).
\end{multline*}  
\end{proof}
\subsection{Proof of Theorem \ref{thm:riesztransform}}
Recall that  the Riesz transform associated with this operator $L$, acting over a function $F\in T^r_r$ (so that $F(\cdot, t)\in L^r(\R^n)$ for almost every $t>0$), is defined by 
$\nabla L^{-\frac{1}{2}}(F(\cdot,t))(x)$
for almost every $t>0$. 
Applying Corollary \ref{cor:riesz}, we obtain, for all $F\in T^r_r$, \begin{multline*}
\|\nabla L^{-\frac{1}{2}}(F)\|_{T^q_r}
\lesssim
\sum_{j\geq 1}4^{-jM}\left(\int_{\R^n}\left(\int_0^{\infty}\int_{B(x,2^{j+1}t)}|F(y,t)|^r\frac{dy\,dt}{t^{n+1}}\right)^{\frac{q}{r}}dx\right)^{\frac{1}{q}} 
\\
+\sum_{k=1}^{M}C_{k,M}
\left(\int_{\R^n}\left(\int_0^{\infty}\left|\mathcal{M}_{p_0}\left(\nabla L^{-\frac{1}{2}}e^{-\frac{kt^2}{2}L}(F(\cdot,t))\right)(x)\right|^r\frac{dt}{t}\right)^{\frac{q}{r}}dx\right)^{\frac{1}{q}} =:I+\sum_{k=0}^{M}C_{k,M}II.
\end{multline*}
Applying \cite[Section 3, Proposition 4]{CoifmanMeyerStein} or \cite{Auscherangles}, but taking $r$ in place of $2$ (the proof is the same), and  taking $M>\frac{n}{\min\{q,r\}}$, we have that
\begin{align*} 
I\lesssim 
\sum_{j\geq 1}4^{-j\left(M-\frac{n}{\min\{q,r\}}\right)}\left(\int_{\R^n}\left(\int_0^{\infty}\int_{B(x,t)}|F(y,t)|^r\frac{dy\,dt}{t^{n+1}}\right)^{\frac{q}{r}}dx\right)^{\frac{1}{q}}
\lesssim \|F\|_{T^q_r}.
\end{align*}
Finally the estimate of $II$ follows by extrapolation. For all weights $w\in A_{\frac{r}{q_-(L)}}\cap RH_{\left(\frac{q_+(L)}{r}\right)'}$ we have that $\nabla L^{-\frac{1}{2}}:L^r(w)\rightarrow L^r(w)$ (\cite[Theorem 5.2]{AuscherMartell:III}) and that $\mathcal{M}_{p_0}:L^r(w)\rightarrow L^r(w)$, for some $p_0>q_-(L)$ 
close enough to $q_-(L)$ so that $w\in A_{\frac{r}{p_0}}$. Besides, we  can also take $r<q_0<q_+(L)$ so that 
$w\in RH_{\left(\frac{q_0}{r}\right)'}$. Using these three facts, applying H\"older's inequality for $\frac{q_0}{r}$, the $L^r(\R^n)-L^{q_0}(\R^n)$ off-diagonal estimates that the semigroup $\{e^{-t^2L}\}_{t>0}$ satisfies (see \cite{Auscher}), and Fubini's theorem, we have that
\begin{align*}
\Bigg(\int_{\R^n}&\int_0^{\infty}
|\mathcal{M}_{p_0}(\nabla L^{-\frac{1}{2}}e^{-\frac{kt^2}{2}L}(F(\cdot,t)))(x)|^{r}
\frac{dt}{t}w(x)dx\Bigg)^{\frac{1}{r}}
\\
&
\lesssim
\left(\int_0^{\infty}\int_{\R^n}
|e^{-\frac{kt^2}{2}L}(F(\cdot,t))(x)|^{r}w(x)dx
\frac{dt}{t}\right)^{\frac{1}{r}}
\\
&
=
\left(\int_0^{\infty}\int_{\R^n}
|e^{-\frac{kt^2}{2}L}(F(\cdot,t))(x)|^{r}
\int_{B(x,t)}w(y)\frac{dy}{w(B(x,t))}
w(x)dx
\frac{dt}{t}\right)^{\frac{1}{r}}
\\
&
=
\left(\int_{\R^n}\int_0^{\infty}\int_{B(y,t)}
|e^{-\frac{kt^2}{2}L}(F(\cdot,t))(x)|^{r}
w(x)\frac{dx}{w(B(x,t))}
\frac{dt}{t}w(y)dy\right)^{\frac{1}{r}}
\\
&
\lesssim
\left(\int_{\R^n}\int_0^{\infty}\int_{B(y,t)}
|e^{-\frac{kt^2}{2}L}(F(\cdot,t))(x)|^{r}
w(x)dx
\frac{dt}{tw(B(y,t))}w(y)dy\right)^{\frac{1}{r}}
\\
&
\lesssim
\left(\int_{\R^n}\int_0^{\infty}\left(\int_{B(y,t)}
|e^{-\frac{kt^2}{2}L}(F(\cdot,t)(x)|^{q_0}
dx\right)^{\frac{r}{q_0}}
\left(\int_{B(y,t)}w(x)^{\left(\frac{q_0}{r}\right)'}dx\right)^{\frac{q_0-r}{q_0}}
\frac{dt}{tw(B(y,t))}w(y)dy\right)^{\frac{1}{r}}
\\
&
\lesssim
\sum_{j\geq 1}e^{-c4^{j}}
\left(\int_{\R^n}\int_0^{\infty}\int_{B(y,2^{j+1}t)}
|F(x,t)|^{r}
dx 
\left(\dashint_{B(y,t)}w(x)^{\left(\frac{q_0}{r}\right)'}dx\right)^{\frac{q_0-r}{q_0}}
\frac{dt}{tw(B(y,t))}w(y)dy\right)^{\frac{1}{r}}
\\
&
\lesssim
\sum_{j\geq 1}e^{-c4^{j}}
\left(\int_{\R^n}\int_0^{\infty}\int_{B(y,2^{j+1}t)}
|F(x,t)|^{r}
\frac{dx\,dt}{t^{n+1}}w(y)dy\right)^{\frac{1}{r}}
\\
&
\lesssim
\left(\int_{\R^n}\int_0^{\infty}\int_{B(y,t)}
|F(x,t)|^{r}
\frac{dx\,dt}{t^{n+1}}w(y)dy\right)^{\frac{1}{r}}.
\end{align*}
The second inequality follows from the fact that $B(y,t)\subset B(x,2t)$ if $x\in B(y,t)$ and from the doubling property of the weight. Then,
$w(B(y,t))\leq w(B(x,2t)\leq 2^{nc_w}w(B(x,t))$.

Therefore, we have that, for all $w\in A_{\frac{r}{q_-(L)}}\cap RH_{\left(\frac{q_+(L)}{r}\right)'}$ and $F\in T^r_{r}$, 
\begin{align*}
\int_{\R^n}\left(\int_0^{\infty}
|\mathcal{M}_{p_0}(\nabla L^{-\frac{1}{2}}e^{-\frac{kt^2}{2}L}(F(\cdot,t)))(x)|^{r}
\frac{dt}{t}\right)^{\frac{r}{r}}w(x)dx
\lesssim
\int_{\R^n}\left(\int_0^{\infty}\int_{B(x,t)}
|F(y,t)|^{r}
\frac{dy\,dt}{t^{n+1}}\right)^{\frac{r}{r}}w(x)dx.
\end{align*}
Recall that $p_{0}$ depended on $w$. But if we now fix $p_{0}>q_{-}(L)$, we have this inequality for  all $w\in A_{\frac{r}{p_{0}}}\cap RH_{\left(\frac{q_+(L)}{r}\right)'}$.
Then, applying \cite[Theorem 3.31]{CruzMartellPerez}, we obtain that, for all $p_{0}<q<q_+(L)$, $w_0\in A_{\frac{q}{p_{0}}}\cap RH_{\left(\frac{q_+(L)}{q}\right)'}$, and all $F\in T^r_{r}$,
\begin{align*}
\int_{\R^n}\left(\int_0^{\infty}
|\mathcal{M}_{p_0}(\nabla L^{-\frac{1}{2}}e^{-\frac{kt^2}{2}L}(F(\cdot,t)))(x)|^{r}
\frac{dt}{t}\right)^{\frac{q}{r}}w_0(x)dx
\lesssim
\int_{\R^n}\left(\int_0^{\infty}\int_{B(x,t)}
|F(y,t)|^{r}
\frac{dy\,dt}{t^{n+1}}\right)^{\frac{q}{r}}w_0(x)dx.
\end{align*}
In particular, if we take $w_0\equiv 1$, we have that $w_0\in A_{\frac{q}{p_{0}}}\cap RH_{\left(\frac{q_+(L)}{q}\right)'}$. Then, for all $p_{0}<r,q<q_+(L)$ and $1\leq k\leq M$, we finally conclude that
$$
II
\lesssim \|F\|_{T^q_r}.
$$
In conclusion, we obtain $\|\nabla L^{-\frac{1}{2}}(F)\|_{T^q_r} \lesssim  \|F\|_{T^q_r}$ for all 
$p_{0}<r,q<q_+(L)$ and all $q_{-}(L)<p_{0}<q_{+}(L)$, and for all $F\in T^r_{r}$. The density of $T^r_{r}\cap T^q_{r}$ in $T^q_{r}$  finishes the proof. 
\qed

\section{Amalgam spaces and generalization}\label{section:amalgam}

Amalgam spaces were first defined by 
Norbert Wiener in 1926, in the formulation of his generalized harmonic analysis. Although,  he considered the particular cases $W(L^1,\ell^2)$ and $W(L^2,L^{\infty})$ in \cite{Wiener}, and $W(L^1,L^{\infty})$ and $W(L^{\infty}, L^1)$ in \cite{Wiener2},
in general, 
for $1\leq p,q<\infty$, the amalgam space $W(L^p,L^q)$ is defined as 
$$
W(L^p,L^q):=\left\{f\in L^{p}_{loc}(\R):\left(\sum_{n\in \Z}\left(\int_{n}^{n+1}|f(t)|^pdt\right)^{\frac{q}{p}}\right)^{\frac{1}{q}}<\infty\right\}.
$$
A significant difference in considering amalgam spaces instead of $L^p$ spaces is that 
amalgam spaces give information about the local, $L^q$, and global, $L^p$, properties of the functions, while $L^p$ spaces do not make that distinction.
 
A generalization of the definition of amalgam spaces for Banach function spaces was done by Feichtinger (see for instance \cite{Feichtinger} and \cite{Feichtinger2}). For $B$ and $C$ Banach function spaces on a locally compact group $G$, satisfying certain conditions, he defined  spaces $W(B,C)$ of distributions. The important thing is that we have equivalence of continuous and discrete norms on those spaces. This has been an important tool in applications. We refer to \cite{Heil} for a deeper discussion on amalgam spaces in the real line. 

 Going on in the historical background of amalgam spaces, we highlight the paper of Holland, \cite{Holland}, that appears to be the first methodical study on amalgam spaces. After that there were important studies on amalgam spaces, for example, by Bertrandias, Datry, and Dupuis \cite{BertrandiasDatryDupuis}, Stewart \cite{Stewart}, and Busby and Smith \cite{BusbySmith}.
For a complete survey on amalgam spaces see \cite{FournierStewart}. 

A natural definition of amalgam spaces in dimension $n\geq 2$ is
$$
(L^p,L^q)(\R^n):=\left\{f\in L^{p}_{loc}(\R^n):\left(\int_{\R^n}\|\chi_{B(x,1)}f\|_{L^p}^q\,dy\right)^{\frac{1}{q}}<\infty\right\}.
$$
Beside, for $1\leq \alpha\leq \infty$,  the subspace $(L^p,L^q)^{\alpha}(\R^n)$ of $(L^p,L^q)(\R^n)$ is defined in \cite{Fofana} by
$$
(L^p,L^q)^{\alpha}(\R^n):=\left\{f\in L^{p}_{loc}(\R^n):\|f\|_{(L^p,L^q)^{\alpha}(\R^n)}<\infty\right\},
$$
where
$$
\|f\|_{(L^p,L^q)^{\alpha}(\R^n)}:=\sup_{r>0}
\left(\int_{\R^n}\left(|B(y,r)|^{\frac{1}{\alpha}-\frac{1}{p}-\frac{1}{q}}\|\chi_{B(x,r)}f\|_{L^p}\right)^q\,dy\right)^{\frac{1}{q}}.
$$
In \cite{AuscherMourgoglou} retracts of tent spaces called slice-spaces, are used. It turns out that they are closely related with amalgam spaces. Let us generalize their definition. For each $t>0$ and $0<p,r<\infty$, the slice-space $(E^p_{r})_{t}$ is defined as the following set:
$$
(E^p_{r})_{t}:=\left\{f\in L^{r}_{loc}(\R^n):\left(\dashint_{B(x,t)}|f(y)|^rdy\right)^{\frac{1}{r}} \in L^p(\R^n)\right\} 
$$
with  
$$
\|f\|_{(E^p_{r})_{t}}= \left(\left(\dashint_{B(x,t)}|f(y)|^rdy\right)^{\frac{p}{r}} \, dx\right)^{\frac{1}{p}}.
$$
Besides, consider the weak slice-spaces 
$$
(wE_r^{p})_{t}:=\left\{f\in L^{r}_{loc}(\R^n):\left(\dashint_{B(x,t)}|f(y)|^rdy\right)^{\frac{1}{r}} \in L^{p,\infty}(\R^n)\right\}.
$$
with 
$$
\|f\|_{(wE^p_{r})_{t}}= \left\|\left(\dashint_{B(x,t)}|f(y)|^rdy\right)^{\frac{p}{r}} \, dx\right\|_{L^{p,\infty}(\R^n)}.
$$
For $1\le r,p<\infty$, note that, for $n=1$, $(E_r^p)_1=W(L^r,L^p)$, and for $n\geq 2$,  $(E_r^p)_1=(L^r,L^p)(\R^n)$. Furthermore, for $p\in [r,\infty)$, since 
$
\|f\|_{(E^p_r)_t}\leq \|f\|_{L^p}$, for all  $t>0$,
$(E_r^p)_t=(L^r,L^p)^p(\R^n)$.

Boundedness on amalgam spaces of the Hardy Littlewood maximal operator,
of Calder\'on-Zygmund operators, of maximal fractional operators, Riesz potentials, etc,  has been studied. See for instance \cite{CLHH}, \cite{CowlingMP}, \cite{FeutoFK}, \cite{Kikuchi}.
From Lemmas \ref{lemma:C_Z}, \ref{lemma:H-Lmaximal} and \ref{lemma:rieszpotential}, and Corollary \ref{cor:riesz}, we can obtain easily boundedness of those operators on slice-spaces, and, therefore,  on amalgam spaces. This  significantly simplifies the previous proofs on this issue.

Let $1\le r<\infty$. Let $t>0$. Consider the applications $i_{t}$ and $\pi_{t}$ in \cite{AuscherMourgoglou}: for $f:\R^n\rightarrow \C$, 
$$
i_{t}(f)(x,s)=f(x)\chi_{[t,et]}(s),
$$
and for $G:\R^{n+1}_+\rightarrow \C$,
$$
\pi_{t}(G)(x)=\int_{t}^{et}G(x,s)\frac{ds}{s}.
$$
It is easy to see that 
$$\pi_{t}\circ i_{t}(f)=f.$$
\begin{lemma} Let $0<p<\infty$ and $1\le r<\infty$. Then $i_{t}:(E_r^p)_t\rightarrow T_r^p$ and $\pi_{t}:T_r^p\rightarrow (E_r^p)_t$ are bounded with the norms being uniform with respect to $t$. In particular, the slice-spaces  $(E_r^p)_t$ are retracts of $T_r^p$. The same happens for the weak slice-spaces, they are retracts of the weak tent spaces. 
\end{lemma}

\begin{proof}
For the slice-spaces when $r=2$, this is observed without proof in  \cite{AuscherMourgoglou}. The proof is the same for all (weak) slice-spaces. It suffices to note that 
$$
\left(\dashint_{B(x,t)}|\pi_{t}(G)(y)|^rdy\right)^{\frac{1}{r}}  \le C \mathcal{A}_r(G)(x)
$$
and 
$$\mathcal{A}_r(i_{t}(f))(x) \le C \left(\dashint_{B(x,et)}|f(y)|^rdy\right)^{\frac{1}{r}}$$
for some dimensional constants $C$,
and to use the norm comparison below for the slice-spaces and similarly for the weak slice-spaces.
\end{proof}

\begin{lemma}\label{lemma:changeofnorms} 
If $0<t,s<\infty$ with $t\sim s$,  $1\le r<\infty$ and $p\in (0,\infty)$, then $(E_r^p)_t=(E_r^p)_s$ with
$$
\|f\|_{(E_r^p)_t}\sim_{n,p}\|f\|_{(E_r^p)_s}
$$
\end{lemma}

For any  linear operator $T$ on functions on $\R^n$, if $\mathcal{T}$ is its extension to functions on $\R^{n+1}$ by tensorisation, then we have 
$T =\pi_{t} \circ \mathcal{T} \circ i_{t}$.  Hence the boundedness of $\mathcal{T}$ carries to $T$ (In the previous theorems, we used the opposite direction: boundedness of $T$ yields boundedness of $\mathcal{T}$. But it was not that immediate). This also applies to maximal operators with easy modifications. So immediate corollaries of our results on tent spaces are the followings.
\begin{proposition}
Let $\mathcal{M}$ be the centered Hardy-Littlewood maximal operator. We have, for all $1<r<\infty$,
\begin{list}{$(\theenumi)$}{\usecounter{enumi}\leftmargin=1cm \labelwidth=1cm\itemsep=0.2cm\topsep=.0cm \renewcommand{\theenumi}{\alph{enumi}}}

\item $\mathcal{M}:(E_r^p)_t\rightarrow (E_r^p)_t$ for all $1<p<\infty.$

\item $\mathcal{M}:(E_r^{1})_t\rightarrow (wE_{r}^{1})_t$\ .

\end{list}
\end{proposition}

\begin{proposition}
Let $\mathcal{T}$ be a Calder\'on-Zygmund operator of order $\delta \in (0,1]$. We have, for all $1<r<\infty$, 
\begin{list}{$(\theenumi)$}{\usecounter{enumi}\leftmargin=1cm \labelwidth=1cm\itemsep=0.2cm\topsep=.0cm \renewcommand{\theenumi}{\alph{enumi}}}

\item $\mathcal{T}:(E_r^p)_t\rightarrow (E_r^p)_t$,\, for all $1<p<\infty.$

\item $\mathcal{T}:(E_r^{1})_t\rightarrow (wE_{r}^{1})_t$.

\item $\mathcal{T}:(\mathfrak{E}_r^p)_t\rightarrow (E_r^p)_t$,\, for all $\frac{n}{n+\delta}<p\leq 1.$

\item $\mathcal{T}:(\mathfrak{E}_r^p)_t\rightarrow (\mathfrak{E}_r^p)_t$,\, for all $\frac{n}{n+\delta}<p\leq 1.$ if $\mathcal{T}^*(1)=0$.

\end{list}
\end{proposition}
Here,
$
(\mathfrak{E}_r^{p})_t
$
is the space of functions in $(E^p_r)_t$ such that there exists an atomic decomposition  $\sum_{i=1}^{\infty}{\lambda_i}a_i$ with $\int_{\R^n}a_i(x)\,dx=0$, for all $i\in \N$. The atoms are defined in \cite{AuscherMourgoglou} for $r=2$ and this adapts here. It suffices for understanding the statement to remark that $(\mathfrak{E}_r^{p})_t= \pi(\mathfrak{T}_r^p)$.

\begin{proposition}
Let $\mathcal{M}_{\alpha}$ be the maximal fractional and $\mathcal{I}_{\alpha}$ the Riesz potential of order $\alpha\in (0,n)$. We have, for all $\frac{n}{n-\alpha}<r<\infty$ and $1<p<q<\infty$ with $\frac{1}{p}-\frac{1}{q}=\frac{\alpha}{n}$, 
$$
\mathcal{M}_{\alpha},\mathcal{I}_{\alpha}:(E_r^q)_t\rightarrow (E_r^p)_t.
$$
\end{proposition}

\begin{proposition}
Let $\nabla L^{-\frac{1}{2}}$ be the Riesz transform associated to $L$. We have, for $q_-(L)<p,r<q_+(L)$,
$$
\nabla L^{-\frac{1}{2}}:(E_r^p)_t\rightarrow (E_r^p)_t.
$$
\end{proposition}
\section{Concluding remarks}

We note that all the arguments using extrapolation prove much more than what we stated. 

For $\frac{n}{n+1}<q<\infty$ and $1<r<\infty$, one can show  that  the set
$$
E:=\left\{\varphi \in C^{\infty}_0(\R^{n+1}_{+}):\, \int_{\R^n}\varphi(x,t)\,dx=0\, \textrm{for all}\,t>0\right\}
$$ 
is dense in $\mathfrak{T}^q_r$ when $q\le 1$ and in ${T}^q_r$ when $q>1$. For $q\le 1$, it suffices to do that on $\mathfrak{T}^q_r$ atoms and for $q>1$, we already know that the space of compactly supported smooth functions in $\R^{n+1}_{+}$ is dense and  those functions can be approximated in $L^r$ norm imposing the  mean value condition   using $r>1$. So the fact that there is a common dense subspace is an indication that the space $\mathfrak{T}^q_r$ is not to small. 

It is clear one can push Theorem \ref{thm:C-Zoperator}, part (c) and (d), to any  Calder\'on-Zygmund operator on $\R^n$ of order $\delta\ge 1$ (see \cite{Duo}, \cite{Grafakos} for definition) imposing more vanishing moments in the definition of $\mathfrak{T}^q_r$ atoms when $q\le \frac{n}{n+1}$ and more cancellation conditions on the adjoint. Similarly, we can play the same game on slice-spaces. These slice-spaces will be subspaces of the classical real Hardy spaces as one can show. We do not insist.

Consider a standard Littlewood-Paley decomposition of $\R^n$ given from a pair of 
$C^\infty_{0}$  functions $\psi,\tilde\psi$ with all vanishing moments and such that
$$
\int_{0}^\infty Q_{t}\tilde Q_{t} f\frac{dt}{t}=f
$$
on appropriate distributions $f$, where $Q_{t}$ and 
$\tilde Q_{t}$ are convolutions with $\psi_{t}$ and $\tilde \psi_{t}$ respectively. We have set $\psi_{t}(x)= t^{-n}\psi(x/t)$ and likewise for $\tilde \psi_{t}$. 
One can show that $f\in H^q(\R^n)$ implies $F(x,t)= \tilde Q_{t}f(x)$ belongs to $\mathfrak{T}^q_2$ and the action is bounded. Conversely, $F\in \mathfrak{T}^q_2$ implies that $f=\int_{0}^\infty Q_{t}F(\cdot, t)\, \frac{dt}{t}$ belongs to $H^q(\R^n)$ and the action is bounded. This is fairly easy to show using atoms and molecules. This can be done for $0<q\le 1$. Thus, $H^q(\R^n)$ can be seen as a  retract of the space  $\mathfrak{T}^q_2$. It is also the case using ${T}^q_2$ instead as shown in \cite{CoifmanMeyerStein}. Nevertheless, the spaces $\mathfrak{T}^q_2$ are preserved by the singular integrals (of convolution) while the ${T}^q_2$ are not. It would be interesting to explore further these spaces (interpolation, etc) and their applications. In particular, one could recover boundedness for Calder\'on-Zygmund operators on tent spaces from interpolation.


\begin{thebibliography}{10}



\bibitem{Amenta}
A. Amenta, {\em Interpolation and embeddings of weighted tent spaces}. arxiv.org/pdf/1509.05699.pdf

\bibitem{Auscher}
P. Auscher, {\em On necessary and sufficient conditions for $L^p$-estimates of Riesz transform associated to elliptic operators on $\R^n$ and related estimates}. Mem. Amer. Math. Soc., {\bf 186} (2007), no. 871.  

\bibitem{Auscherangles}
P. Auscher, {\em Change of Angles in tent spaces,} 
Comptes rendus - Math\'ematique. {\bf 349} (2011), no. 5, 297--301 


\bibitem{AuscherHofmannMartell}
P. Auscher, S. Hofmann, J.M. Martell, {\em Vertical versus conical square functions}. Trans. Amer. Math. Soc. \textbf{364} (2012), no. 10, 5469-5489.

\bibitem{AuscherMartell:II}
P. Auscher, J.M. Martell, {\em Weighted norm inequalities,
off-diagonal estimates and elliptic operators. Part II: off-diagonal
estimates on spaces of homogeneous type}, J. Evol. Equ. {\bf 7} (2007),
no. 2, 265-316.

\bibitem{AuscherMartell:III}
P. Auscher, J.M. Martell, {\em Weighted norm inequalities,
off-diagonal estimates and elliptic operators. Part III: harmonic analysis of elliptic operators}. J. Funct. Anal. {\bf 241} (2006), 703--746.

\bibitem{AMcR} 
{\sc Auscher, P.,  McIntosh, A., and Russ, E.}
\newblock Hardy spaces of differential forms on Riemannian manifolds. 
\newblock {\em J. Geom. Anal.} {\bf 18} 1, (2008), 192-248.

\bibitem{AKMP}
P. Auscher, C. Kriegler, S. Monniaux, P. Portal, {\em 
Singular integral operators on tent spaces}, J. Evol. Equ.
{\bf 12} (2012), no. 4, 741-765.


\bibitem{AuscherMourgoglou}
P. Auscher, M. Mourgoglou, {\em Representation and uniqueness for boundary
value elliptic problems via first order systems},  arXiv:1404.2687.


\bibitem{BertrandiasDatryDupuis}
J.-P. Bertrandias, C. Datry, and C. Dupuis, {\em Unions et intersections d'espaces $L^p$ invariantes par translation ou convolution}. Ann. Inst. Fourier (Grenoble), {\bf  28} (1978), pp. 53-84.

\bibitem{BusbySmith}
R.C. Busby and H.A. Smith, {\em Product-convolution operators and mixed -norm spaces}. Trans. Amer. Math.Soc., {\bf 263} (1981), pp. 309-341.


\bibitem{CLHH}
C. Carton-Lebrun, H. Heinig, S. Hofmann, {\em Integral operators on weighted
amalgams} , Studia Math {\bf 109} (1994), 133-157.


\bibitem{CoifmanFefferman}
R.R. Coifman, C. Fefferman, 
{\em Weighted norm inequalities for maximal functions and singular integrals}, Studia Mathematica, T. Li. (1974).


\bibitem{CoifmanMeyerStein} 
R.R. Coifman, Y. Meyer, E.M. Stein, {\em Some new function spaces
and their applications to harmonic analysis}. J. Funct. Anal.
62(2), 304-335 (1985).

\bibitem{CoifmanWeiss}
R.R. Coifman, G. Weiss, {\em Extensions of Hardy spaces and their use in analysis}, Bull.
Amer. Math. Sot. 83 (1977), 569-645.

\bibitem{CowlingMP}
M. Cowling, S. Meda, R. Pasquale,
{\em Riesz potentials and amalgams}, Ann. Inst. Fourier, Grenoble
{\bf 49}, 4 (1999), 1345-1367.

\bibitem{CruzMartellPerez}
D.V. Cruz-Uribe, J.M. Martell, C. Perez, {\em Weights Extrapolation and the Theory of Rubio de Francia}. Operator Theory: Advances and Applications, Vol. 215.

\bibitem{Duo} J.~Duoandikoetxea, {\em Fourier Analysis},
Grad. Stud. Math.\ 29, American Math.\ Soc., Providence, 2000.


\bibitem{FeffermanStein}
C. Fefferman, E. M. Stein, {\em Some maximal inequalities}. Amer. J. Math. {\bf 93} (1971), 107-116.

\bibitem{FeutoFK}
J. Feuto, I. Fofana, K. Koua,
{\em Weighted norm inequalities for a maximal operator in some subspace of amalgams}, http://arxiv.org/pdf/0901.4197.pdf.


\bibitem{Fofana}
I. Fofana, {\em Continuit\'e de l'int\'egral fractionnaire et espace $(L^p,\ell^q)^{\alpha}$}. Comptes Redus de S\'eances de l'Acad\'emie de Sciences I, vol. 308, pp. 525-527, 1989.


\bibitem{FournierStewart}
J.J.F. Fournier, J. Stewart, {\em
Amalgams of $L^p$ and $\ell^q$}. Bull. Amer. Math. Soc. (N.S.) {\bf 13} (1985), 1-21.


\bibitem{Feichtinger}
H.G. Feichtinger,
{\em Generalized amalgams, with applications to Fourier transform}. Canad. J. Math. 42(3), 395-409, (1990).

\bibitem{Feichtinger2}
H.G. Feichtinger,
{\em Banach convolution algebras of Wiener type}. Functions, Series, Operators, Proc. Conf. Budapest, 38, Colloq.
Math. Soc. J\'anos Bolyai (1980), 509 524.

\bibitem{GCRF85} J.\,Garc{\'\i}a-Cuerva and J.\,Rubio de Francia, {\it Weighted Norm
Inequalities and Related Topics}, North Holland, Amsterdam, 1985.



\bibitem{Grafakos}
L. Grafakos, {\em Modern Fourier Analysis, second edition}. Graduate Texts in Mathematics, Springer.



\bibitem{Heil}
C.E. Heil, {\em Wiener amalgam spaces in generalized
harmonic analysis and
wavelet theory}, Ph.D.
Thesis, University of Maryland, College  Park, MD, (1990).
ftp://ftp.math.gatech.edu/pub/users/heil/thesis.pdf.

\bibitem{HofmannMayboroda}
S. Hofmann, S. Mayboroda, {\em Hardy and $BMO$ spaces to divergence
form elliptic operators.} Math. Ann. {\bf 344} (2009), no. 1, 37-116.

\bibitem{HofmannMayborodaMcIntosh}
S. Hofmann, S. Mayboroda, A. McIntosh, {\em Second order elliptic operators with complex bounded measurable coefficients in $L^p$, Sobolev and Hardy spaces.} Ann. Sci. \'Ecole. Norm. Sup. (4), 44(5):723--800, 2011.

\bibitem{Holland}
F. Holland,
{\em Harmonic analysis on amalgams of $L^p$ and $\ell^q$,}
J. London Math. Soc. (2), {\bf 10}(1975), pp. 295-305.


\bibitem{HNP}
T. Hyt\"onen, J. van Neerven, P. Portal {\em Conical square function estimates in UMD Banach spaces and applications to $H^{\infty}-$functional calculi.} J. Anal. Math. {\bf 106} (2008), 317-351.

\bibitem{Kikuchi}
N. Kikuchi, E. Nakai, N. Tomita, K. Yabuta,
T. Yoneda,
{\em Calder\'on-Zygmund operators on amalgam spaces
and in the discrete case}. J. Math. Anal. Appl. {\bf 335} (2007) 198-212.

\bibitem{KT}
H.~Koch and D.~Tataru,
\newblock {\em Well-posedness for the Navier-Stokes equations.}
\newblock {Adv. Math.}, 157(1):22--35, 2001. 


\bibitem{MartellPrisuelos}
J.M. Martell, C. Prisuelos-Arribas {\em Weighted norm inequalities for conical square functions. Part: I}. To appear in Trans. Amer. Math. Soc.

\bibitem{Muckenhoupt}
B. Muckenhoupt, {\em Weighted norm inequalities for the HArdy maximal function}. Trans. Amer. Soc. 165 (1972), pp. 207-226. 


\bibitem{MuckenhouptWheede}
B. Muckenhoupt, R.L. Wheeden, {\em Weighted norm inequalities for fractional integrals}. Trans. Amer. Math. Soc. vol. 192, 1974.

\bibitem{Stein}
E.M. Stein, {\em Singular integrals and differentiability properties of functions}. Princeton University Press, Princeton, New Jersey, 1970. 

\bibitem{Stewart}
J. Stewart, {\em Fourier transforms of unbounded measures}, Canad. J. Math., {\bf 31} (1979), pp. 1281-1292.

\bibitem{Torchinsky}
A. Torchinsky, {\em Real-variable methods in harmonic analysis}. Academic Press, INC, 1986.

\bibitem{Wiener}
N. Wiener, {\em On the representation of functions by trigonometric integrals}. Math. Z. {\bf 24}, 1926, 575-616.


\bibitem{Wiener2}
N. Wiener, {\em  Tauberian
 theorems,}
 Ann. of  Math. 33 (1932), 1-100. 


\end{thebibliography}
\end{document}